\documentclass[a4paper,11pt]{amsart}
\title{Flops and the S-duality conjecture}
\date{}
\author{Yukinobu Toda}

\usepackage{amscd}
\usepackage{amsmath}
\usepackage{amssymb}
\usepackage{amsthm}
\usepackage{float}
\usepackage[dvips]{graphicx}

\usepackage[all,ps,dvips]{xy}

\dedicatory{To the memory of Kentaro Nagao}

\usepackage{array}
\usepackage{amscd}
\usepackage[all]{xy}

\DeclareFontFamily{U}{rsfs}{%
\skewchar\font127}
\DeclareFontShape{U}{rsfs}{m}{n}{%
<-6>rsfs5<6-8.5>rsfs7<8.5->rsfs10}{}
\DeclareSymbolFont{rsfs}{U}{rsfs}{m}{n}
\DeclareSymbolFontAlphabet
{\mathrsfs}{rsfs}
\DeclareRobustCommand*\rsfs{%
\@fontswitch\relax\mathrsfs}

\theoremstyle{plain}
\newtheorem{thm}{Theorem}[section]
\newtheorem{prop}[thm]{Proposition}
\newtheorem{lem}[thm]{Lemma}

\newtheorem{defi}[thm]{Definition}
\newtheorem{rmk}[thm]{Remark}
\newtheorem{cor}[thm]{Corollary}

\newtheorem{prop-defi}[thm]{Proposition-Definition}
\newtheorem{thm-defi}[thm]{Theorem-Definition}
\newtheorem{lem-defi}[thm]{Lemma-Definition}

\newtheorem{exam}[thm]{Example}
\newtheorem{assume}[thm]{Assumption}

\newdimen\argwidth
\def\db[#1\db]{
 \setbox0=\hbox{$#1$}\argwidth=\wd0
 \setbox0=\hbox{$\left[\box0\right]$}
  \advance\argwidth by -\wd0
 \left[\kern.3\argwidth\box0 \kern.3\argwidth\right]}

\newcommand{\bB}{\mathcal{B}}
\newcommand{\cC}{\mathcal{C}}
\newcommand{\dD}{\mathcal{D}}
\newcommand{\eE}{\mathcal{E}}
\newcommand{\fF}{\mathcal{F}}

\newcommand{\hH}{\mathcal{H}}

\newcommand{\lL}{\mathcal{L}}
\newcommand{\mM}{\mathcal{M}}

\newcommand{\oO}{\mathcal{O}}
\newcommand{\pP}{\mathcal{P}}

\newcommand{\sS}{\mathcal{S}}

\newcommand{\uU}{\mathcal{U}}

\newcommand{\xX}{\mathcal{X}}
\newcommand{\yY}{\mathcal{Y}}

\newcommand{\Supp}{\mathop{\rm Supp}\nolimits}
\newcommand{\Hom}{\mathop{\rm Hom}\nolimits}

\newcommand{\dR}{\mathbf{R}}

\newcommand{\Hilb}{\mathop{\rm Hilb}\nolimits}

\newcommand{\ch}{\mathop{\rm ch}\nolimits}

\newcommand{\td}{\mathop{\rm td}\nolimits}
\newcommand{\Ext}{\mathop{\rm Ext}\nolimits}
\newcommand{\Spec}{\mathop{\rm Spec}\nolimits}

\newcommand{\Coh}{\mathop{\rm Coh}\nolimits}

\newcommand{\cneq}{\mathrel{\raise.095ex\hbox{:}\mkern-4.2mu=}}
\newcommand{\eqcn}{\mathrel{=\mkern-4.5mu\raise.095ex\hbox{:}}}

\newcommand{\Aut}{\mathop{\rm Aut}\nolimits}

\newcommand{\Ex}{\mathop{\rm Ex}\nolimits}

\newcommand{\oPPer}{\mathop{\rm ^{0}Per}\nolimits}

\newcommand{\iPPer}{\mathop{\rm ^{-1}Per}\nolimits}

\newcommand{\pGamma}{\mathop{^{{p}}\Gamma}\nolimits}
\newcommand{\oGamma}{\mathop{^{{0}}\Gamma}\nolimits}
\newcommand{\iGamma}{\mathop{^{{-1}}\Gamma}\nolimits}

\newcommand{\pA}{\mathop{^{p}\hspace{-.2em}\mathcal{A}}\nolimits}

\newcommand{\pB}{\mathop{^{p}\hspace{-.05em}\mathcal{B}}\nolimits}
\newcommand{\oB}{\mathop{^{0}\hspace{-.05em}\mathcal{B}}\nolimits}
\newcommand{\iB}{\mathop{^{-1}\hspace{-.05em}\mathcal{B}}\nolimits}
\newcommand{\pPhi}{\mathop{^{p}\hspace{-.05em}\Phi}\nolimits}

\newcommand{\pF}{\mathop{^{p}\hspace{-.05em}\mathcal{F}}\nolimits}
\newcommand{\iF}{\mathop{^{-1}\hspace{-.05em}\mathcal{F}}\nolimits}
\newcommand{\oF}{\mathop{^{0}\hspace{-.05em}\mathcal{F}}\nolimits}

\newcommand{\pT}{\mathop{^{p}\hspace{-.05em}\mathcal{T}}\nolimits}

\newcommand{\oT}{\mathop{^{0}\hspace{-.05em}\mathcal{T}}\nolimits}

\newcommand{\pE}{\mathop{^{p}\hspace{-.05em}\mathcal{E}}\nolimits}
\newcommand{\oE}{\mathop{^{0}\hspace{-.05em}\mathcal{E}}\nolimits}

\newcommand{\pAA}{\mathop{^{p}\hspace{-.2em}A}\nolimits}

\newcommand{\pDT}{\mathop{^{p}\hspace{-.0em}\widehat{\rm DT}}\nolimits}
\newcommand{\oDT}{\mathop{^{0}\hspace{-.0em}\widehat{\rm DT}}\nolimits}
\newcommand{\iDT}{\mathop{^{-1}\hspace{-.0em}\widehat{\rm DT}}\nolimits}

\newcommand{\pPPer}{\mathop{\rm ^{\mathit{p}}Per}\nolimits}

\newcommand{\DT}{\mathop{\rm DT}\nolimits}

\newcommand{\modu}{\mathop{\rm mod}\nolimits}

\newcommand{\Imm}{\mathop{\rm Im}\nolimits}

\newcommand{\length}{\mathop{\rm length}\nolimits}

\begin{document}
\maketitle

\begin{abstract}
We prove the transformation formula of 
Donaldson-Thomas (DT) invariants counting
two dimensional torsion sheaves on
Calabi-Yau 3-folds
 under flops. 
The error term is described by the Dedekind eta function
and the Jacobi theta function, and our result
gives evidence of 
a 3-fold version of Vafa-Witten's 
S-duality conjecture. 
As an application, 
we prove a 
blow-up formula of DT type invariants 
on the total spaces of canonical 
line bundles on smooth projective 
surfaces. It gives an analogue 
of the similar blow-up formula in the 
original S-duality conjecture by Yoshioka, Li-Qin
and G\"ottsche. 
\end{abstract}

\section{Introduction}
\subsection{Background and motivation}
The purpose of this paper is to give 
evidence of a 3-fold version of Vafa-Witten's S-duality 
conjecture~\cite{VW}. 
The original S-duality conjecture
predicts the (at least almost)
modularity of 
the generating series of Euler characteristics of 
moduli spaces of stable torsion free sheaves on algebraic surfaces. 
This is still an open problem, but there exist several evidence,
and we refer to~\cite{Goinv} for the developments so far. 
One of the important evidence is 
that the above generating series transform 
under a blow-up at a point of a surface 
by a multiplication of a certain 
modular form. 
Such a blow-up formula was predicted by Vafa-Witten~\cite{VW}, 
and proved by Yoshioka~\cite{Yo1}, 
Li-Qin~\cite{WZ} and 
G\"ottsche~\cite{GoTheta}.  

Instead of stable torsion free sheaves on algebraic surfaces, 
we study semistable pure two dimensional torsion sheaves on 
3-folds
and the generating series of 
their counting invariants. 
Let $X$ be a smooth projective 
Calabi-Yau 3-fold, i.e. 
$K_X=0$ and $H^1(\oO_X)=0$. 
For 
an ample divisor $\omega$ on $X$ and 
a cohomology class $v \in H^{\ast}(X, \mathbb{Q})$, 
we have the 
\textit{(generalized) Donaldson-Thomas (DT) 
invariant} (cf.~\cite{Thom},~\cite{JS},~\cite{K-S})
\begin{align}\label{intro:DTv}
\DT_{\omega}(v) \in \mathbb{Q}
\end{align}
 which virtually counts 
$\omega$-semistable sheaves $E \in \Coh(X)$ 
satisfying $v(E)=v$. Here $v(E)$ is the \textit{Mukai vector}
\footnote{
Taking the Mukai vector, 
rather than the 
usual Chern character, is crucial 
for our purpose. 
See Remark~\ref{rmk:Mukai}.} of $E$:
\begin{align*}
v(E) \cneq \ch(E) \cdot \sqrt{\td}_X \in H^{\ast}(X, \mathbb{Q}). 
\end{align*}  
We are interested in the DT invariants (\ref{intro:DTv})
for the classes $v$ of the form 
\begin{align*}
v=(0, P, -\beta, -n)
\in H^0(X) \oplus H^2(X) \oplus H^4(X) \oplus H^6(X).
\end{align*}
Below we regard $\beta, n$ as elements of
$H_2(X)$, $\mathbb{Q}$ via Poincar\'e duality. 
By fixing $P$, 
we consider the generating 
series
\begin{align}\label{intro:DT}
\sum_{\beta\in H_2(X), n \in \mathbb{Q}}
\DT_{\omega}(0, P, -\beta, -n)q^n t^{\beta}. 
\end{align}
The sheaves which contribute to the series (\ref{intro:DT}) are
supported on two dimensional subschemes in $X$. 
The S-duality conjecture
for our 3-fold $X$
claims that the series (\ref{intro:DT}) satisfies 
(almost) modular transformation properties of Jacobi 
forms\footnote{We
refer to~\cite{EZ} for a basic of Jacobi forms.}. 
Such an expected modularity 
of the series (\ref{intro:DT})
plays an important role in~\cite{DM}
to derive the Ooguri-Strominger-Vafa (OSV)
conjecture~\cite{OSV} in string 
theory.
The OSV conjecture, 
which is not yet 
formulated in a mathematically precise way, 
may be stated as a certain approximation between 
the series (\ref{intro:DT}) for $\lvert P \rvert \gg 0$
and 
the generating series of Gromov-Witten invariants on $X$. 
The study of the 
S-duality conjecture for the series (\ref{intro:DT})
is an important subject toward
a mathematical approach of the OSV conjecture. 

In general, 
 computing the generating series (\ref{intro:DT})
in concrete examples
may be 
more difficult than those which 
appear in the S-duality conjecture for smooth 
surfaces, 
due to the possible singularities in 
the supports of two dimensional sheaves 
on 3-folds. 
Nevertheless, we can ask how the generating series (\ref{intro:DT}) transforms
under a birational transformation, and whether 
the error term is described in terms of Jacobi forms or not. 
This is an important process
 to check the validity of the 
S-duality conjecture for the series (\ref{intro:DT}), as well as 
the blow-up formula in~\cite{Yo1},~\cite{WZ},~\cite{GoTheta}
toward the
S-duality conjecture for surfaces.

\subsection{Main result}
Our main result shows that a
variant of
the generating series (\ref{intro:DT})
transforms 
under a flop 
by a multiplication of a certain meromorphic 
Jacobi form, 
which gives evidence of the
3-fold version of 
S-duality conjecture. 
Let $X$ be a smooth projective 
Calabi-Yau 3-fold
which fits into a flop diagram (cf.~Definition~\ref{defi:flop})
\begin{align}\label{intro:flop}
\xymatrix{
(C\subset X) \ar[dr]_{f}  \ar@{.>}[rr]^{\phi} &  & (X^{\dag} \supset C^{\dag})
 \ar[dl]^{f^{\dag}} \\
&  (p\in Y).  &
}
\end{align}
Here $C, C^{\dag}$ are exceptional 
loci of $f$, $f^{\dag}$
and $f(C)=f(C^{\dag})=p$. 
We assume that $C$ is an irreducible rational 
curve and let $l$ be the 
scheme theoretic length of $f^{-1}(p)$ at the generic point 
of $C$. 
For a fixed divisor class $P \in H^2(X)$
and an ample divisor $\omega$ on $Y$, 
we consider the generating series\footnote{
The invariant $\DT_{f^{\ast}\omega}(v)$
 counts
$f^{\ast}\omega$-slope semistable sheaves $E$
on $X$
with $v(E)=v$, which is well-defined 
although $f^{\ast}\omega$ is not ample.}
\begin{align}\label{intro:con}
\DT_{f^{\ast}\omega}(P) \cneq
\sum_{\beta\in H_2(X), n \in \mathbb{Q}}
\DT_{f^{\ast}\omega}(0, P, -\beta, -n)q^n t^{\beta}. 
\end{align}
 The following theorem, 
which will be proved in Subsection~\ref{subsec:proof}, 
is the main result in this paper: 
\begin{thm}\label{thm:intro}
There exist
$n_j \in \mathbb{Z}_{\ge 1}$ for $1\le j\le l$ 
such that we have the following formula: 
\begin{align*}
&\DT_{f^{\dag \ast} \omega}(\phi_{\ast}P)
=\phi_{\ast}\DT_{f^{\ast}\omega}(P)
 \\
&\hspace{20mm}\cdot 
\prod_{j=1}^{l} \left\{ 
i^{jP \cdot C-1}
\eta(q)^{-1}\vartheta_{1, 1}(q, ((-1)^{\phi_{\ast}P}t)^{jC^{\dag}}) \right\}^{jn_j P \cdot C}. 
\end{align*}
Here $\phi_{\ast}$ is the variable change
$(n, \beta) \mapsto (n, \phi_{\ast}\beta)$, 
$\eta(q)$ is the Dedekind eta function and 
$\vartheta_{a, b}(q, t)$ is the Jacobi theta function, 
given as follows:
\begin{align}\label{thetaab}
\eta(q)= q^{\frac{1}{24}}\prod_{k\ge 1}(1-q^k), \quad 
\vartheta_{a, b}(q, t) =
\sum_{k \in \mathbb{Z}}
q^{\frac{1}{2}\left(k+\frac{a}{2} \right)^2}\{(-1)^b t\}^{k+\frac{a}{2}}. 
\end{align}
\end{thm}
Recall that $\eta(q)$ is a modular form of weight $1/2$, 
$\vartheta_{1, 1}(q, t)$ is a Jacobi form of weight $1/2$ and index $1/2$. 
The result of Theorem~\ref{thm:intro} shows that 
the series (\ref{intro:con})
transforms under a flop 
by a multiplication of a meromorphic 
Jacobi form of weight $0$ and index
$\sum_{j=1}^{l} j^3 n_j (P \cdot C)/2$. 
The integer $n_j$ is defined to be the multiplicity 
of the Hilbert scheme on $X$ at some one dimensional 
subscheme in $X$ with class $j[C]$
(cf.~Definition~\ref{def:nj}). 
It also coincides with 
the genus zero Gopakumar-Vafa invariant 
with class $j[C]$ in the sense of~\cite{Katz}. 

We will also prove some variants of Theorem~\ref{thm:intro}:
the Euler characteristic version and a 
version with fixed supports. 
The former version 
is a flop formula for DT type invariants without 
Behrend functions~\cite{Beh}. 
Such invariants are not deformation invariant, 
and the computation of the error 
term of the flop formula is more subtle.  
We will give that version only in the case of $l=1$, 
while we don't need the Calabi-Yau condition. 
The latter version compares DT type invariants 
on $X$ and $X^{\dag}$, which count 
sheaves supported on a fixed divisor $S \subset X$ and 
its strict transform $S^{\dag} \subset X^{\dag}$
respectively. 
The above variants give interesting applications. 
For instance, we give a blow-up formula 
of DT type invariants on canonical line bundles on 
surfaces in Section~\ref{section:blow}. 
In the next paper~\cite{TodS2}, we 
also apply the result in this paper to 
show the modularity of the generating series of 
Hilbert schemes of points with $A_n$-type singularities.

\subsection{Blow-up formula}
As an application of a variant of Theorem~\ref{thm:intro}, 
we prove a blow-up formula 
for DT type invariants without Behrend functions
on canonical line bundles on surfaces, 
which gives an analogue of the similar blow-up formula 
in the original S-duality conjecture~\cite{Yo1}, \cite{WZ}, \cite{GoTheta}. 
Let $S$ be a smooth projective surface and $\pi \colon 
\omega_S \to S$
the total space of the 
canonical line bundle of $S$, which is a non-compact 
Calabi-Yau 3-fold. 
For an ample divisor $L$ on $S$ and 
an element
\begin{align*}
(r, l, s)\in H^0(S) \oplus H^2(S) \oplus H^4(S)
\end{align*}
the invariant
\begin{align}\label{DTL}
\DT_{L}^{\chi}(r, l, s)\in \mathbb{Q}
\end{align} 
is defined to be (roughly speaking)
the naive Euler characteristic of 
the moduli space of $L$-semistable sheaves 
$E \in \Coh(\omega_S)$ 
supported on the zero section of $\pi$
such that the Chern character of $\pi_{\ast}E$ coincides with 
$(r, l, s)$. 
Let $g \colon S^{\dag} \to S$ be a blow-up at a point in $S$
with exceptional curve $C^{\dag} \subset S^{\dag}$. 
The following result will be proved in Theorem~\ref{thm:blow}:
\begin{thm}\label{thm:blow:intro}
For fixed $r \in \mathbb{Z}_{\ge 1}$ and 
$l\in H^2(S)$, we have the following formula:
\begin{align*}
&\sum_{s, a} \DT^{\chi}_{g^{\ast}L}(r, g^{\ast}l-aC^{\dag}, -s)
q^{\frac{r}{12}+\frac{a}{2}+s}t^{a+\frac{r}{2}} \\
&\hspace{30mm} =\sum_{s} \DT^{\chi}_{L}(r, l, -s)q^{s} 
\cdot \eta(q)^{-r} 
\vartheta_{1, 0}(q, t)^{r}. 
\end{align*}
\end{thm}
It is easy to see that (cf.~Example~\ref{exam:r2}
for the rank two case)
the error term of the 
formula in Theorem~\ref{thm:blow:intro}
coincides with the error term
in the blow-up formula in the original S-duality 
conjecture~\cite{Yo1},~\cite{WZ},~\cite{GoTheta}.
The idea of the proof of Theorem~\ref{thm:blow:intro}
is 
as follows: we find a suitable flop diagram (\ref{intro:flop})
such that 
$X$, $X^{\dag}$ are compactifications
of $\omega_S$, $\omega_{S^{\dag}}$ respectively. 
Moreover the strict transform of the zero section 
$S \subset \omega_S \subset X$ coincides with 
the zero section $S^{\dag} \subset \omega_{S^{\dag}} \subset X^{\dag}$. 
Then a variant of 
Theorem~\ref{thm:intro}
compares the invariants (\ref{DTL}) on $S$ and $S^{\dag}$, 
which proves the result.

\subsection{Sketch of the proof of Theorem~\ref{thm:intro}}
Here we give a rough sketch of the proof of Theorem~\ref{thm:intro}.
Let $f \colon X \to Y$
be a 3-fold flopping contraction
as in Theorem~\ref{thm:intro}. 
By Bridgeland~\cite{Br1}, 
there exist hearts of perverse t-structures 
\begin{align*}
\pPPer_{\le 2}(X/Y) \subset D^b \Coh_{\le 2}(X)
\end{align*}
for $p=0, -1$. 
Here $\Coh_{\le 2}(X)$ is the category of torsion sheaves on $X$. 
We introduce the $f^{\ast}\omega$-slope stability 
on $\pPPer_{\le 2}(X/Y)$ (cf.~Definition~\ref{defi:mu}), 
and construct another heart 
\begin{align*}
\pA_{f^{\ast}\omega}^{\mu} \subset D^b \Coh_{\le 2}(X)
\end{align*}
by a tilting of 
$\pPPer_{\le 2}(X/Y)$ determined by the $f^{\ast}\omega$-slope stability
on it and a choice of $\mu\in \mathbb{Q}$. 
Moreover we construct the abelian subcategory
(cf.~Subsection~\ref{subsec:Tilt})
\begin{align*}
\pB_{f^{\ast}\omega}^{\mu} \subset \pA_{f^{\ast}\omega}^{\mu}
\end{align*}
which consists of objects in $\pA_{f^{\ast}\omega}^{\mu}$
whose Chern characters satisfy a 
certain linear equation. It has the following properties:
first for a given element $v \in H^{\ast}(X)$, 
the set of objects in $\pB_{f^{\ast}\omega}^{\mu}$
with Mukai vector $v$ is bounded (cf.~Proposition~\ref{prop:bound}). 
Hence we are able to define the 
completion of the stack
theoretic Hall algebra of $\pB_{f^{\ast}\omega}^{\mu}$, 
denoted by $\widehat{H}(\pB_{f^{\ast}\omega}^{\mu})$.
Next we see
that the category 
$\pB_{f^{\ast}\omega}^{\mu}$
contains any $f^{\ast}\omega$-slope semistable 
two dimensional sheaf on $X$ 
with slope $\mu$ (cf.~Lemma~\ref{EinCoh}). 
Hence the moduli spaces of such objects, 
which define the series (\ref{intro:con}), 
 determine an element of 
$\widehat{H}(\pB_{f^{\ast}\omega}^{\mu})$. 

For a flop diagram (\ref{intro:flop}),
there is an equivalence $\Phi$
between $D^b \Coh(X)$ and $D^b \Coh(X^{\dag})$
by Bridgeland~\cite{Brs1}. 
The equivalence $\Phi$ restricts to an equivalence 
between
$\oB_{f^{\ast}\omega}^{\mu}$ and $\iB_{f^{\dag \ast}\omega}^{\mu}$. 
Noting this fact, we construct the generating series
\begin{align}\label{intro:DTh}
\pDT_{f^{\ast}\omega}(P)
=\sum_{\beta, n} \pDT_{f^{\ast}\omega}(0, P, -\beta, -n)q^n t^{\beta}
\end{align}
whose coefficients satisfy 
\begin{align}\label{oDTi}
\oDT_{f^{\ast}\omega}(v)=\iDT_{f^{\dag\ast}\omega}(\Phi_{\ast}v). 
\end{align} 
Here $\Phi_{\ast}$ is 
the isomorphism between $H^{\ast}(X)$ and $H^{\ast}(X^{\dag})$
induced by $\Phi$, which is described in Lemma~\ref{lem:psi}. 

In Proposition~\ref{prop:delta}, 
we describe the relationship 
between moduli spaces which define (\ref{intro:con})
and (\ref{intro:DTh}) in terms of the algebra
 $\widehat{H}(\pB_{f^{\ast}\omega}^{\mu})$. 
Together with the 
integration map on the Lie algebra of 
virtual indecomposable objects in $\widehat{H}(\pB_{f^{\ast}\omega}^{\mu})$
by~\cite{JS}, 
it 
enables us to describe the 
relationship between the series (\ref{intro:con})
and (\ref{intro:DTh}). 
Combined with (\ref{oDTi}),  
we describe in Theorem~\ref{thm:main}
the relationship between the series
$\phi_{\ast}\DT_{f^{\ast}\omega}(P)$
and $\DT_{f^{\dag \ast}\omega}(\phi_{\ast}P)$.
The error term is described
in terms of DT type invariants
counting 
one dimensional semistable sheaves supported on 
the fibers of $f^{\dag} \colon X^{\dag}\to Y$. 
It is also described in terms of 
invariants counting
\textit{parabolic stable pairs} introduced in~\cite{Todpara}. 
We compute the error term by using the deformation 
invariance of these invariants
(cf.~Proposition~\ref{prop:comp}). 
In proving the Euler characteristic version of Theorem~\ref{thm:intro}, 
we give a direct classification of parabolic stable pairs 
when $l=1$ (cf.~Lemma~\ref{lem:compv}).

\subsection{Related works}
A flop formula for DT type 
curve counting invariants 
was 
obtained in the papers~\cite{HL}, ~\cite{NN}, ~\cite{Tcurve2}, \cite{Cala}. 
Among them, the papers~\cite{Tcurve2}, \cite{Cala}
(also see~\cite{BrH})
use similar Hall algebra methods, 
but we need to work with the relevant abelian category 
$\pB_{f^{\ast}\omega}^{\mu}$
which did not appear in the above papers. 
On the other hand, 
there are few mathematical literatures in which DT invariants
of the form
$\DT_{\omega}(0, P, \beta, n)$ are studied. 
In~\cite{GS3}, the modularity of these invariants is 
discussed for nodal K3 fibrations using degeneration formula. 
In~\cite{TodBG}, \cite{GS2}, some relationship between
the invariants $\DT_{\omega}(0, P, \beta, n)$ and 
DT type curve counting invariants are
studied. In~\cite{GS1}, 
the invariant $\DT_{\omega}(0, P, \beta, n)$ on local $\mathbb{P}^2$
is studied for small $P$. 
In physics literatures,
a few of the 
D4 brane counting which corresponds to the 
invariants of the form $\DT_{\omega}(0, P, \beta, n)$
are computed~\cite{GSY06}, \cite{GX07}.
Also  
the flop formula of D4D2D0 bound states
on the resolved conifold is studied in 
\cite{Nishi}, \cite{NiYa}
using Kontsevich-Soibelman's wall-crossing formula~\cite{K-S}. 
The result of Theorem~\ref{thm:intro} is interpreted as
a
mathematical justification and
 a generalization of the 
arguments in the physics articles~\cite{Nishi}, \cite{NiYa}.  
\subsection{Acknowledgment}
This paper is dedicated to the memory 
of Kentaro Nagao, who made significant 
contributions to the Donaldson-Thomas theory 
and its relationship to flops, non-commutative 
algebras~\cite{NN}, \cite{Nagao},
 during his short life. 
I would like to 
thank the referees for reading the manuscript carefully, and 
giving valuable comments. 
This work is supported by World Premier 
International Research Center Initiative
(WPI initiative), MEXT, Japan. This work is also supported by Grant-in Aid
for Scientific Research grant (22684002)
from the Ministry of Education, Culture,
Sports, Science and Technology, Japan.
\subsection{Notation and convention}
In this paper, all 
the varieties are defined over $\mathbb{C}$. 
For a $d$-dimensional variety $X$, we denote by $H^{\ast}(X, \mathbb{Q})$ 
the even part of the singular cohomologies of $X$, and
write its element as 
$(a_0, a_1, \cdots, a_d)$ for $a_i \in H^{2i}(X, \mathbb{Q})$.
We sometimes abbreviate $\mathbb{Q}$ and just write 
$H^{2i}(X, \mathbb{Q})$, $H_{2i}(X, \mathbb{Q})$
as $H^{2i}(X)$, $H_{2i}(X)$. 
For subschemes $Z_1, Z_2 \subset X$, 
the intersection $Z_1 \cap Z_2 \subset X$ always 
means the scheme theoretic intersection. 
For a triangulated category $\dD$ and a set of objects
$S$ in $\dD$, we denote by $\langle S \rangle_{\rm{ex}}$
the smallest extension closed subcategory in $\dD$
which contains $S$.
For a variety $X$
and a sheaf of (possibly non-commutative)
algebras $A$ on $X$, 
we denote by $\Coh(A)$
the abelian category of coherent right $A$-modules on $X$, 
 and $D^b \Coh(A)$ its
bounded derived category. 
We write $\Coh(\oO_X)$ as $\Coh(X)$ as usual. 
For $i\in \mathbb{Z}$, we denote by 
$\Coh_{\le i}(A) \subset \Coh(A)$
the subcategory of objects $E\in \Coh(A)$ whose 
support $\Supp(E)$ as an $\oO_X$-module satisfies
$\dim \Supp(E) \le i$. 
For $E \in D^b \Coh(A)$, we denote by 
$\hH^i(E) \in \Coh(A)$ its $i$-th 
cohomology. 

\section{Flops, perverse t-structures and their tilting}
\subsection{3-fold flops}\label{subsec:flop}
Let us recall the notion of flopping contractions and their flops. 
\begin{defi}\label{defi:flop}
A projective birational morphism $f \colon X \to Y$
is called a flopping contraction 
if $f$ is isomorphic in codimension one, $Y$
has only Gorenstein singularities, and 
the relative Picard number of $f$ equals to one. 
A flop of a flopping contraction 
$f \colon X \to Y$ is a non-isomorphic birational 
morphism 
$\phi \colon X \dashrightarrow X^{\dag}$ which fits into 
a commutative diagram
\begin{align}\label{flop:dia}
\xymatrix{
X \ar[dr]_{f}  \ar@{.>}[rr]^{\phi} &  & X^{\dag} \ar[dl]^{f^{\dag}} \\
&  Y  &
}
\end{align}
such that $f^{\dag}$ is also a flopping contraction. 
\end{defi}
It is known that
a flop is unique if it exists, and 
any birational map between 
minimal models is decomposed into a finite 
number of flops~\cite{Kawaflo}. 
We say that $f \colon X \to Y$ is a 
\textit{3-fold flopping contraction} if 
$f$ is a flopping contraction and 
$X$ is a smooth 3-fold.  
In this case, the
 exceptional locus $C$ of $f$ 
is a tree of smooth rational curves
\begin{align}\label{Exf}
C=C_1 \cup \cdots \cup C_N, \ C_i \cong \mathbb{P}^1. 
\end{align}
If $f^{\dag} \colon X^{\dag} \to Y$ is a flop of 
$f$, the exceptional locus $C^{\dag}$ of $f^{\dag}$ is 
also a tree of smooth 
rational curves $C_i^{\dag}$ with $1\le i \le N$. 
A projective line on a smooth 3-fold 
is called $(a, b)$-\textit{curve} if its normal bundle 
is isomorphic to $\oO_{\mathbb{P}^1}(a) \oplus \oO_{\mathbb{P}^1}(b)$. 
It is well-known (cf.~\cite[Section~5]{Rei})
that each $C_i$ is an $(a, b)$-curve 
for either $(a, b)=(-1, -1)$, $(0, -2)$ or $(1, -3)$. 
\begin{exam}\label{exam}
Let $Y \subset \mathbb{C}^4$ be 
the 3-fold singularity, given by
\begin{align}\label{cAn}
Y=\{ xy+z^2 -w^{2n}=0 : (x, y, z, w) \in \mathbb{C}^4\}. 
\end{align}
Then there is a flop diagram (\ref{flop:dia}), 
where $f$, $f^{\dag}$ are blow-ups at the 
ideals
\begin{align*}
I=(x, z-w^n) \subset \oO_Y, \quad 
I^{\dag}=(x, z+w^n) \subset \oO_Y
\end{align*}
respectively. 
Both of the exceptional loci $C$, $C^{\dag}$
of $f$, $f^{\dag}$
are
$(-1, -1)$-curves
if $n=1$, and $(0, -2)$-curves otherwise. 
By~\cite[Section~5]{Rei}, 
the birational map 
$\phi$ is obtained as a Pagoda diagram
\begin{align}\label{Pagoda}
X \stackrel{f_1}{\leftarrow} X_1 \stackrel{f_2}{\leftarrow}
 \cdots \stackrel{f_{n-1}}{\leftarrow} 
X_{n-1} \stackrel{f_n}{\leftarrow} X_n \stackrel{f_n^{\dag}}{\to}
 X_{n-1}^{\dag} \stackrel{f_{n-1}^{\dag}}{\to} \cdots 
\stackrel{f_2^{\dag}}{\to} X_1^{\dag} 
\stackrel{f_1^{\dag}}{\to} 
X^{\dag}. 
\end{align}
Here $f_i$, $f_i^{\dag}$ for $1\le i\le n-1$
are blow-ups at 
$(0, -2)$-curves, and $f_{n}$, $f_n^{\dag}$ are blow-ups
at $(-1, -1)$-curves. 
\end{exam}

Let $f \colon X \to Y$ be a 3-fold flopping contraction
whose exceptional locus is an irreducible 
rational curve $C \subset X$. 
We denote by $l$ the length of $\oO_{f^{-1}(p)}$ at the 
generic point of $C$, where $p=f(C)$ and 
$f^{-1}(p)$ is the scheme theoretic fiber of 
$f$ at $p$. 
Then we have 
\begin{align*}
l \in \{1, 2, 3, 4, 5, 6\}
\end{align*}
and 
$l=1$ if and only if $C$ is not a $(1, -3)$-curve
(cf.~\cite[Section~1]{KaMo}). 
Moreover if $l=1$, then
$\widehat{\oO}_{Y, p}$ is isomorphic to 
the completion of the singularity (\ref{cAn}) 
for some $n\in \mathbb{Z}_{\ge 1}$
at the origin (cf.~\cite[Remark~5.3 (b)]{Rei}). 
In this case, the integer $n$ 
is called the \textit{width} of $C$. 

Let $p\in S \subset Y$ be a general hyperplane,  
and $\overline{S} \subset X$ its proper transform. 
Then we have $C \subset \overline{S}$. 
Let $I \subset \oO_{\overline{S}}$ be
the ideal sheaf of $C$. 
For $1\le j \le l$, we have the subscheme
$C^{(j)} \subset \overline{S}$ 
whose structure sheaf is given by 
$(\oO_{\overline{S}}/I^j)/Q$, where 
$Q$ is the maximum zero dimensonal subsheaf of 
$\oO_{\overline{S}}/I^j$. 
Let $\Hilb(X)$ be the Hilbert scheme of subschemes in 
$X$. 
In~\cite[Section~2.1]{BKL}, 
it is shown that 
$C^{(j)}$ is the isolated point in $\Hilb(X)$, and 
the following number is defined: 
\begin{defi}\label{def:nj}
For $1\le j\le l$, we define $n_j \in \mathbb{Z}_{\ge 1}$
to be
\begin{align*}
n_j \cneq \length \oO_{\Hilb(X), C^{(j)}}. 
\end{align*}
By convention, we define $n_j=0$ for $j>l$. 
\end{defi}
Since $\oO_{\Hilb(X), C^{(j)}}$ is 
a finitely generated Artinian $\mathbb{C}$-algebra, 
the length of $\oO_{\Hilb(X), C^{(j)}}$
is well-defined. 
If $l=1$, the number $n_1$ equals to 
the width $n$.
In general, the number $n_j$ appears
in the context of deformations in the following way. 
Let us consider the following completion: 
\begin{align}\label{complet}
\widehat{f} \colon 
\widehat{X} \cneq X \times_Y \Spec \widehat{\oO}_{Y, p} \to 
\widehat{Y} \cneq \Spec \widehat{\oO}_{Y, p}. 
\end{align}
By~\cite[Section~2.1]{BKL},
there exists a flat deformation 
\begin{align}\label{deform}
\xymatrix{ 
\widehat{\xX} \ar[r]^{h}\ar[dr] & \widehat{\yY} \ar[d] \\
&  \Delta
}
\end{align} 
where $\Delta$ is 
a Zariski open neighborhood of 
$0\in \mathbb{A}^1$
such that $h_0 \colon \widehat{\xX}_0 \to \widehat{\yY}_0$ 
is isomorphic to $\widehat{f}$, 
and $h_t \colon \widehat{\xX}_t \to \widehat{\yY}_t$
for $t \in \Delta \setminus \{0\}$ 
is a flopping contraction whose 
exceptional locus is a disjoint union of
$(-1, -1)$-curves. 
Then the number $n_j$ coincides with 
the number of $h_t$-exceptional 
$(-1, -1)$-curves $C' \subset \widehat{\xX}_t$
for $t\neq 0$ whose class is $j[C']$, i.e. 
for any line bundle $\lL$ on $\widehat{\xX}$, we have
\begin{align*}
\deg(\lL|_{C'})=j \deg(\lL|_{C})
\end{align*}
 where 
we regard $C$ as a curve on the central fiber of 
$\widehat{\xX} \to \Delta$. 
In what follows, we write 
the exceptional locus of $h_t$ for $t\neq 0$ 
as
\begin{align*}
C_{j, k} \subset \widehat{\xX}_t, \ 
1\le j\le l, \ 1\le k\le n_j
\end{align*}
where $C_{j, k}$ is a $(-1, -1)$-curve with class $j[C]$.

\subsection{Perverse t-structures}\label{subsec:perv}
Let $f \colon X \to Y$
be a 3-fold flopping contraction. 
In this situation, Bridgeland~\cite[Section~3]{Br1}
associates the subcategories 
$\pPPer(X/Y) \subset D^b \Coh(X)$ for 
$p=0, -1$ as follows: 
\begin{defi}
We define $\pPPer(X/Y) \subset D^b \Coh(X)$
for $p=0, -1$ to be
\begin{align*}
\pPPer(X/Y) \cneq \left\{ E \in D^b \Coh(X) : 
\begin{array}{c}
\dR f_{\ast} E \in \Coh(Y) \\
\Hom^{<-p}(E, \cC)=\Hom^{<p}(\cC, E)=0 
\end{array} \right\}. 
\end{align*}
Here $\cC\cneq \{ F \in \Coh(X) : \dR f_{\ast} F=0\}$. 
\end{defi}
We also define $\pPPer_{\le i}(X/Y)$
to be 
\begin{align*}
\pPPer_{\le i}(X/Y) \cneq \{ E \in \pPPer(X/Y) :
\dR f_{\ast} E \in \Coh_{\le i}(Y) 
\}. 
\end{align*}
It is proved in~\cite[Section~3]{Br1} that $\pPPer(X/Y)$ 
are the hearts of bounded t-structures on the 
category $D^b \Coh(X)$. 
Similarly for $i>0$, 
the categories $\pPPer_{\le i}(X/Y)$
are the hearts of bounded t-structures
on $D^b \Coh_{\le i}(X)$, hence 
they are abelian categories. 
This fact immediately follows 
from the argument of~\cite[Section~3]{Br1}, noting 
that $\cC \subset D^b \Coh_{\le i}(X)$ for $i>0$ 
and $\dR f_{\ast}$ takes 
$D^b \Coh_{\le i}(X)$ to $D^b \Coh_{\le i}(Y)$. 

Let us consider the $i=0$ case. 
In this case, 
$\pPPer_0(X/Y) \cneq \pPPer_{\le 0}(X/Y)$
is not a subcategory of $D^b \Coh_{\le 0}(X)$. 
Instead, we consider the derived category of 
the following abelian category
\begin{align*}
\Coh_0(X/Y) \cneq \{ F \in \Coh_{\le 1}(X) : 
f_{\ast} F \in \Coh_{\le 0}(F) \}. 
\end{align*}
We set
\begin{align}\label{pTF}
\pT\cneq\pPPer_0(X/Y) \cap \Coh(X), \ 
\pF\cneq \pPPer_0(X/Y)[-1] \cap \Coh(X).
\end{align}
By~\cite{MVB}, 
the pairs of categories $(\pT, \pF)$
form torsion pairs (cf.~\cite{HRS})
of $\Coh_0(X/Y)$
such that $\pPPer_0(X/Y)$ is the 
associated tilting:
\begin{align}\label{Per0}
\pPPer_0(X/Y)=\langle \pF[1], \pT \rangle_{\rm{ex}}. 
\end{align}
In particular, the categories 
$\pPPer_0(X/Y)$ 
are the hearts of bounded t-structures on $D^b \Coh_0(X/Y)$. 
The torsion pairs $(\pT, \pF)$
are also described in terms of 
semistable sheaves. 
For $F \in \Coh_0(X/Y)$
and an ample divisor $H$ on $X$, we set
\begin{align}\label{muH}
\mu_H(F) \cneq \frac{\ch_3(F)}{\ch_2(F) \cdot H}. 
\end{align}
Here $\mu_H(F)$ is set to be
$\infty$ if the denominator is zero.
The above slope function defines the $\mu_H$-stability 
on $\Coh_0(X/Y)$ in the usual way. 
\begin{lem}\label{desc:F}
We have the following descriptions of $\pF$: 
\begin{align*}
\oF&=
\langle E \in \Coh_0(X/Y) : E \mbox{ is }\mu_H 
\mbox{-semistable with }
\mu_H(E) < 0 \rangle_{\rm{ex}} \\
\iF&=
\langle E \in \Coh_0(X/Y) : E \mbox{ is }\mu_H 
\mbox{-semistable with }
\mu_H(E) \le 0 \rangle_{\rm{ex}}. 
\end{align*}
\end{lem}
\begin{proof}
The result essentially follows from~\cite[Proposition~5.2 (iii)]{ToBPS}. 
For simplicity, suppose that $p=0$. 
For $B+iH \in H^2(X, \mathbb{C})$
and $F \in D^b \Coh_0(X/Y)$, we set
\begin{align*}
Z_{B, H}(F) \cneq -\ch_3(F)+ (B+iH)\ch_2(F).
\end{align*}
Then
by~\cite[Proposition~5.2 (iii)]{ToBPS},  
the pair 
\begin{align*}
\sigma_0 \cneq (Z_{-\delta H, 0}, \oPPer_0(X/Y)), \quad 
0<\delta \ll 1
\end{align*}
determines a Bridgeland stability condition~\cite{Brs1} 
on $D^b \Coh_0(X/Y)$. 
Note that $Z_{-\delta H, 0}(F)$
for any non-zero $F \in \oPPer_0(X/Y)$ is contained in 
$\mathbb{R}_{<0}$. 
Moreover
the pair 
\begin{align*}
\sigma_t \cneq (Z_{-\delta H, tH}, \Coh_0(X/Y)), \quad 
0<t\ll 1
\end{align*}
gives a small deformation of $\sigma_0$
in the space of Bridgeland stability conditions on $D^b \Coh_0(X/Y)$. 
Let $\pP_t(\phi) \subset D^b \Coh_0(X/Y)$
be the full subcategory for each $\phi \in \mathbb{R}$, 
which gives the slicing in the sense of~\cite{Brs1}
corresponding to $\sigma_t$. 
Let us take $E \in \pP_t(\phi)$ with 
$1/2 <\phi \le 1$.
We have $E \in \Coh_0(X/Y)$, 
and there is an exact sequence
\begin{align*}
0 \to T \to E \to F \to 0
\end{align*}
with $T\in \oT$ and $F \in \oF$. 
Since $F[1] \in \oPPer_0(X/Y) \cap (\Coh_0(X/Y)[1])$, 
it follows that $\arg Z_{-\delta H, tH}(F) \in (0, \pi/2)$.
Hence the $Z_{-\delta H, tH}$-semistability of $E$
yields $F=0$, i.e. $E \in \oT$. 
This implies that, 
in the notation of~\cite{Brs1},
we have $\pP_t(1/2, 1]\subset \oT$, 
and a similar argument also shows that 
$\pP_t(0, 1/2) \subset \oF$.
Since both of 
$(\oT, \oF)$ and $(\pP_t(1/2, 1], \pP_t(0, 1/2))$
are torsion pairs of $\Coh_0(X/Y)$, they must coincide. 
Note that an object $E \in \Coh_0(X/Y)$ is 
contained in $\pP_t(\phi)$ for $0<\phi<1/2$
and $0<t \ll 1$
if and only if $E$ is $\mu_H$-semistable 
with $\mu_H(E) < -\delta$. 
Since we can take $\delta>0$ arbitrary close to zero, 
we obtain the result.  
\end{proof}
Let $\phi \colon X \dashrightarrow X^{\dag}$ be the 
flop of $f$. 
By~\cite{Br1}, there is an equivalence 
\begin{align}\label{Deq}
\Phi \colon D^b \Coh(X) \stackrel{\sim}{\to} D^b \Coh(X^{\dag})
\end{align}
which takes $\oPPer(X/Y)$ to $\iPPer(X^{\dag}/Y)$. 
Furthermore the equivalence $\Phi$ is given 
by the Fourier-Mukai functor with kernel 
$\oO_{X \times_Y X^{\dag}}$ (cf.~\cite{Ch}), 
hence $\Phi$ also takes 
$D^b \Coh_{\le i}(X)$ to $D^b \Coh_{\le i}(X^{\dag})$ 
for $i>0$ and 
$\oPPer_{\le i}(X/Y)$ to $\iPPer_{\le i}(X^{\dag}/Y)$
for $i\ge 0$. 
\begin{lem}\label{lem:OtoO}
We have $\Phi(\oO_X) \cong \oO_{X^{\dag}}$. 
\end{lem}
\begin{proof}
The object $\oO_X \in \oPPer(X/Y)$ is a local 
projective object in $\oPPer(X/Y)$
by~\cite[Lemma~3.2.4]{MVB}, 
hence $\Phi(\oO_X) \in \iPPer(X^{\dag}/Y)$ is a 
local projective object. 
By~\cite[Proposition~3.2.6]{MVB}, 
the object $\Phi(\oO_X)$ must be a line bundle, 
hence it must be isomorphic to $\oO_{X^{\dag}}$. 
\end{proof}
The abelian categories $\pPPer(X/Y)$ are 
also related to sheaves of non-commutative algebras. 
By~\cite{MVB},
there are vector bundles $\pE$ on $X$
which admit equivalences
\begin{align}\label{PhiA}
\pPhi \cneq \dR f_{\ast} \dR \hH om(\pE, \ast) 
\colon D^b \Coh(X) \stackrel{\sim}{\to}
D^b \Coh(\pAA_Y)
\end{align}
where $\pAA_Y \cneq f_{\ast} \eE nd(\pE)$ are sheaves of non-commutative 
algebras on $Y$. 
The equivalence $\Phi$ restrict to equivalences 
between $\pPPer(X/Y)$ and $\Coh(\pAA_Y)$.

\subsection{Induced morphism on cohomologies}
For an object $E \in D^b \Coh(X)$, its Mukai 
vector is defined by 
\begin{align*}
v(E) \cneq \ch(E) \cdot \sqrt{\td}_X \in H^{\ast}(X). 
\end{align*}
The $H^{2i}(X)$-component of $v(E)$ is
denoted by $v_i(E)$. 
Let $\Phi$ be the derived equivalence (\ref{Deq}). 
By the Grothendieck Riemann-Roch theorem,  
there is a commutative diagram (\cite[Section~2.3]{Cal2})
\begin{align}\label{com:dia}
\xymatrix{
D^b \Coh(X) 
\ar[r]^{\Phi} \ar[d]_{v}  &  D^b \Coh(X^{\dag}) \ar[d]^{v} \\
H^{\ast}(X, \mathbb{Q}) \ar[r]^{\Phi_{\ast}} & H^{\ast}(X^{\dag}, \mathbb{Q}).
}
\end{align}
Here $\Phi_{\ast}$ is defined by the 
correspondence $v(\oO_{X\times_Y X^{\dag}})$. 
The map $\Phi_{\ast}$ is an isomorphism, which 
does not preserve the grading. 
However since $v(\oO_{X\times_Y X^{\dag}})$
is written as
\begin{align*}
v(\oO_{X\times_Y X^{\dag}})=[X \times_Y X^{\dag}]
+(\mbox{elements in } H^{\ge 7}(X \times X^{\dag})), 
\end{align*}
the map $\Phi_{\ast}$ takes 
$H^{\ge i}(X)$
to $H^{\ge i}(X^{\dag})$
for any $i\in \mathbb{Z}$. 
Therefore $\Phi_{\ast}$ induces the 
graded isomorphism
\begin{align}\label{gisom}
\bigoplus_{i\in \mathbb{Z}} H^{\ge i}(X)/H^{\ge i+1}(X)
\stackrel{\sim}{\to} 
\bigoplus_{i\in \mathbb{Z}} 
H^{\ge i}(X^{\dag})/H^{\ge i+1}(X^{\dag}). 
\end{align}
We denote by $\phi_{\ast} \colon H^{\ast}(X) \stackrel{\cong}{\to} 
H^{\ast}(X^{\dag})$
the graded isomorphism (\ref{gisom}). 
The isomorphism $\phi_{\ast}$ is 
given by the correspondence 
$[X \times_Y X^{\dag}]$. 

We set $\Gamma$ to be
\begin{align*}
\Gamma \cneq \Imm \left( v(\ast) \colon D^b \Coh_{\le 2}(X) \to 
H^{\ge 2}(X, \mathbb{Q})  \right). 
\end{align*}
We identify 
$H^4(X), H^6(X)$ with $H_2(X), 
H_0(X) \cong \mathbb{Q}$
via Poincar\'e duality
respectively. 
We write an element of $\Gamma$
as $(P, \beta, n)$ for 
$P \in H^2(X)$, $\beta \in H_2(X)$ and $n\in \mathbb{Q}$. 
Let $H_2(X/Y)$ be the kernel
of 
\begin{align*}
f_{\ast} \colon H_2(X) \to H_2(Y).
\end{align*}
We use the following description of 
the action of $\Phi_{\ast}$ on $\Gamma$: 
\begin{lem}\label{lem:linear}
There exist
linear maps 
\begin{align*}
\psi_1 \colon H^2(X) \to H_2(X/Y), \ 
\psi_0 \colon H^2(X) \to \mathbb{Q}
\end{align*}
such that 
we have 
\begin{align}\label{Phist}
\Phi_{\ast}(P, \beta, n)=(\phi_{\ast}P, 
\phi_{\ast}\beta+\psi_1(P), n+\psi_0(P))
\end{align}
for any $(P, \beta, n) \in \Gamma$. 
\end{lem}
\begin{proof}
The result follows from~\cite[Proposition~5.2]{ToBPS}. 
The precise descriptions of $\psi_0$ and $\psi_1$ will 
be given in Lemma~\ref{lem:linear2} below. 
 \end{proof}
For $\beta \in H_2(X)$, we write 
$\beta >0$ if 
$\beta$ is a numerical class of a non-zero
 effective one cycle on $X$. 
We have the following lemma: 
\begin{lem}\label{lem:cup}
For any divisor class $P \in H^2(X)$
and curve class $\beta \in H_2(X)$, we have 
$P \cdot \beta=\phi_{\ast}P \cdot \phi_{\ast}\beta$. 
In particular, $\beta \in H_2(X/Y)$
satisfies $\beta>0$ if and only if 
$\phi_{\ast}\beta <0$ in $H_2(X^{\dag}/Y)$. 
\end{lem}
\begin{proof}
Let us take $E \in \Coh_{\le 2}(X)$ and $F \in \Coh_{\le 1}(X)$. 
Since $\Phi$ is an equivalence, 
we have 
$\chi(E, F)=\chi(\Phi(E), \Phi(F))$. 
By the Riemann-Roch theorem, the LHS coincides with 
$-\ch_1(E) \cdot \ch_2(F)$. 
By Lemma~\ref{lem:linear}, the 
RHS coincides with $-\phi_{\ast}\ch_1(E) \cdot \phi_{\ast} \ch_2(F)$. 
Hence the result follows. 
As for the second statement, let $P$ be an $f$-ample divisor. 
Then for $\beta \in H_2(X/Y)$, we have 
$\beta>0$ if and only if $P \cdot \beta>0$, 
since $f$ is a flopping contraction. 
Since $-\phi_{\ast}P$ is $f^{\dag}$-ample, 
the equality
$P \cdot \beta=\phi_{\ast}P \cdot \phi_{\ast}\beta$
implies that $\phi_{\ast}\beta<0$. 
\end{proof}
In the following lemma, we give some more 
precise descriptions of $\psi_0$, $\psi_1$:
\begin{lem}\label{lem:linear2}
There exists
 $a\in \mathbb{Q}$
such that we have 
$\psi_1(P)=a(P\cdot C_1)C_1^{\dag}$ and 
\begin{align}\label{ab}
\psi_0(P)=\frac{1}{24}\left(c_2(X) - \phi_{\ast}^{-1}c_2(X^{\dag})\right)P
\end{align}
for any divisor class $P$.
\end{lem}
\begin{proof}
Obviously $\psi_i(f^{\ast}D)=0$ for any 
Cartier divisor $D$ on $Y$.
Hence we can write 
$\psi_1(P)=a(P \cdot C_1)C_1^{\dag}$ and 
$\psi_0(P)=b(P \cdot C_1)$
for some $a, b\in \mathbb{Q}$. 
Here we have used that every $C_i$ is numerically proportional 
to $C_1$
 by the definition of a flopping 
contraction. 
In order to obtain (\ref{ab}), 
it is enough to prove this 
for $P=[S]$ for any irreducible divisor $S \subset X$. 
Since $\Phi(\oO_X)=\oO_{X^{\dag}}$ by Lemma~\ref{lem:OtoO}, 
we have $\chi(\oO_X, \oO_S)=\chi(\oO_{X^{\dag}}, \Phi(\oO_S))$. 
By the Riemann-Roch theorem, 
we have
\begin{align}\label{RR1}
\chi(\oO_X, \oO_S)
=\frac{P^3}{6}+ \frac{c_2(X)}{12}P.
\end{align}
On the other hand, the 
commutative 
diagram (\ref{com:dia})
implies
\begin{align*}
v(\Phi(\oO_S))=\left(\phi_{\ast}P, -\phi_{\ast} \frac{P^2}{2}, 
\frac{P^3}{6}+\frac{c_2(X)}{24}P+\psi_0(P)   \right). 
\end{align*}
Using the Riemann-Roch theorem again, we obtain
\begin{align}\label{RR2}
\chi(\oO_{X^{\dag}}, \Phi(\oO_S))
=
\frac{P^3}{6}+ \frac{c_2(X)}{24}P
+ \psi_0(P)+\frac{c_2(X^{\dag})}{24}\phi_{\ast}P. 
\end{align}
Comparing (\ref{RR1}) with (\ref{RR2}), and using 
Lemma~\ref{lem:cup}, we obtain (\ref{ab}). 
\end{proof}

If the exceptional locus
of $f$ is an irreducible rational curve, the linear maps
$\psi_0$, $\psi_1$ are described
using the integers 
$n_1, \cdots, n_l$ in Subsection~\ref{subsec:flop}. 
\begin{prop}\label{lem:psi}
Suppose that 
the exceptional locus $C$ of $f$ is an irreducible rational 
curve. Then we have 
\begin{align*}
\psi_0(P)=-\frac{1}{12}\sum_{j=1}^{l} jn_j(P\cdot C), 
\quad 
\psi_1(P)=-\frac{1}{2}\sum_{j=1}^{l} j^2 n_j(P\cdot C)C^{\dag}. 
\end{align*}
\end{prop}
\begin{proof}
It is enough to prove the claim for 
one divisor class $P$ with $P \cdot C \neq 0$. 
We first prove the claim 
in the case that 
$C$ is a $(-1, -1)$-curve 
and 
there is a smooth surface $S \subset X$
such that $S \cap C$ consists of a reduced
one point, e.g. 
a suitable compactification of a flopping 
contraction in Example~\ref{exam} with $n=1$. 
In this case, 
$l=n_1=1$, and 
the flop $\phi$ is 
obtained as 
\begin{align*}
X \stackrel{g}{\leftarrow} Z \stackrel{g^{\dag}}{\rightarrow}
 X^{\dag}
\end{align*}
where $g$, $g^{\dag}$ are blow-ups
at $C$, $C^{\dag}$ respectively. 
Let $E$ be the exceptional locus of $g$, 
which is isomorphic to $\mathbb{P}^1 \times \mathbb{P}^1$. 
Let $l_1, l_2$ be the lines in $E$ which 
are contracted to points by $g$, $g^{\dag}$ respectively. 
By the blow-up formula of Chern classes (cf.~\cite{Fu}), 
we have 
\begin{align*}
c_2(Z)=g^{\ast}c_2(X)+l_2 -l_1 = g^{\dag \ast}c_2(X^{\dag})
+l_1 -l_2
\end{align*}
which shows $c_2(X)-\phi_{\ast}^{-1}c_2(X^{\dag})=
-2C$. 
By (\ref{ab}), we obtain the desired formula for $\psi_0$ in 
this case. 
On the other hand, since
the equivalence $\Phi$ 
coincides with $\dR g^{\dag}_{\ast} \circ g^{\ast}$ 
in this case, 
we have $\Phi(\oO_S) \cong \oO_{S^{\dag}}$, 
where $S^{\dag}$ is the strict transform of $S$. 
The commutative diagram (\ref{com:dia})
shows that  
\begin{align*}
\psi_1(S)=\frac{1}{2}\phi_{\ast}(S^2) -\frac{1}{2}S^{\dag 2}. 
\end{align*}
By the base point free theorem for $f \colon X \to Y$, 
there is a divisor $S'$ on $X$ which is linearly equivalent 
to $S$ such that $S' \cap C$ consists of a
reduced one point which is different 
from $S \cap C$. 
The intersection
$S^{\dag} \cap S^{'\dag}$ contains 
$C^{\dag}$ as a connected component, which is reduced. 
Hence we have
 $S^{\dag 2}-\phi_{\ast}(S^2)=C$, 
and we obtain the desired formula for $\psi_1$ in this case. 

Next we prove the general case. 
Let $S\subset X$ be an irreducible 
divisor which is sufficiently ample
and does not contain $C$. 
We have
\begin{align*}
\Phi(\oO_{S}) \cong 
\dR p_{X^{\dag}\ast} \left( \oO_{S \times_Y X^{\dag}}  \right)
\end{align*}
which is a sheaf since 
$S \times_Y X^{\dag} \to X^{\dag}$ is finite onto its image. 
Here $p_{X^{\dag}}$ is the projection from 
$X\times_Y X^{\dag}$ to $X^{\dag}$. 
Let $F$ be the sheaf obtained as the cokernel of the natural 
injection $\oO_{S^{\dag}} \to p_{X^{\dag}\ast}\oO_{S \times_Y X^{\dag}}$. 
Note that $F$ is supported on $C^{\dag}$, 
and it depends only on the pair $(S, X)$
restricted to 
the completion (\ref{complet}). 
Using (\ref{Phist}), (\ref{ab}) and noting $\ch_3(F)=\chi(F)$, 
we obtain 
\begin{align}\label{form:psi}
\psi_0(S)&=\frac{1}{12}S^{\dag 3}-\frac{1}{12}S^3 + \frac{1}{2}\chi(F) \\
\notag
\psi_1(S)&=\frac{1}{2}\phi_{\ast}(S^2) -\frac{1}{2}S^{\dag 2} +[F]. 
\end{align}
Let $S' \subset X$ be another divisor which 
is linearly equivalent to $S$ and does not contain $C$. 
We can take $S'$ so that $S \cap S'\cap C=\emptyset$. 
Then we have $S^{'\dag 2}-\phi_{\ast}S^2= c C^{\dag}$, 
where $c$ is the length of $S^{\dag} \cap S^{'\dag}$ at the generic 
point of $C$, and 
$S^{\dag 3}-S^3=-c(S \cdot C)$
by Lemma~\ref{lem:cup}. 
Hence $\psi_i(S)$ depends only on the
data $(S, S', X)$
restricted to the completion $\widehat{X}$, 
denoted by $(\widehat{S}, \widehat{S}', \widehat{X})$.  
In particular, the result holds for any 
3-fold flopping contraction which 
contracts a $(-1, -1)$-curve to a point. 

Let $h \colon \widehat{\xX} \to \widehat{\yY}$ 
be a deformation as in (\ref{deform}), 
and take its flop $h^{\dag} \colon \widehat{\xX}^{\dag} \to 
\widehat{\yY}$.  
Since $H^2(\oO_{\widehat{X}})=0$, 
by shrinking $\Delta$ if necessary, 
the divisors $\widehat{S}, \widehat{S}'$ deform to $h$-ample 
divisors $\widehat{\sS}, \widehat{\sS'} \subset \widehat{\xX}$
which are flat over $\Delta$.
Let $\fF$ be the sheaf on $\widehat{\xX}^{\dag}$
obtained as the cokernel of 
the injection $\oO_{\widehat{\xX}^{\dag}} \to 
p_{\widehat{\xX}^{\dag} \ast} \oO_{\widehat{\sS} \times_{\widehat{\yY}} 
\widehat{\xX}^{\dag}}$, 
where $p_{\widehat{\xX}^{\dag}}$ is the projection 
from $\widehat{\xX} \times_{\widehat{\yY}} \widehat{\xX}^{\dag}$
to $\widehat{\xX}^{\dag}$. 
The sheaf $\fF$ is a flat deformation of $F$, and for
 $t \in \Delta \setminus \{0\}$ the restriction $\fF_t =\fF|_{\xX_t}$
decomposes into the direct sum of 
$\fF_{t, k, j}$ where $\fF_{t, k, j}$ is supported on 
$C_{k, j}^{\dag}$. 
Also let $\widehat{\sS}^{\dag}$, $\widehat{\sS}^{'\dag} \subset
\widehat{\xX}^{\dag}$
be the strict transforms of $\widehat{\sS}$, $\widehat{\sS}'$
respectively. 
The intersection $\widehat{\sS}^{\dag} \cap \widehat{\sS}^{'\dag}$
is a flat deformation of 
$\widehat{S} \cap \widehat{S}'$. 
For $t\in \Delta \setminus \{0\}$, 
the fundamental cycle of
 $(\widehat{\sS}^{\dag} \cap \widehat{\sS}^{'\dag})_t=
\widehat{\sS}_t^{\dag} \cap \widehat{\sS}^{'\dag}_t$
is written as 
$\sum_{j, k}c_{j, k}C_{j, k}^{\dag}$
 for some $c_{j, k} \in \mathbb{Z}_{\ge 0}$. 
By (\ref{form:psi}), 
the result for the $(-1, -1)$-flopping contractions shows that
\begin{align*}
&-\frac{1}{12}c_{j, k}(\widehat{\sS}_t \cdot C_{j, k})+ \frac{1}{2}
\chi(\fF_{t, j, k})=-\frac{1}{12}(\widehat{\sS}_t \cdot C_{j, k}) \\
&-\frac{1}{2}c_{j, k} C_{j, k}^{\dag} + 
[\fF_{t, j, k}] = -\frac{1}{2}(\widehat{\sS}_t \cdot C_{j, k})C_{j, k}^{\dag}. 
\end{align*}
Note that $c=\sum_{j, k}jc_{j, k}$, 
$[F]=\sum_{j, k}[\fF_{t, j, k}]$, $[C_{j, k}]=j[C]$
 and $[C_{j, k}^{\dag}]=j[C^{\dag}]$. 
By taking the sum for all $k, j$, we obtain the desired formula for 
$\psi_0, \psi_1$. 
\end{proof}

\subsection{Tilting via slope stability conditions}\label{subsec:Tilt}
Let $X$ be a smooth projective 3-fold and 
$f \colon X \to Y$ a flopping contraction.
Let $\lL_Y$ be an ample line bundle on $Y$
with first Chern class $\omega$. 
We consider the $f^{\ast}\omega$-slope stability conditions 
on $\Coh_{\le 2}(X)$ and $\pPPer_{\le 2}(X/Y)$ 
based on the slope stability conditions
for torsion sheaves 
in~\cite[Definition~1.6.8]{Hu} which use
$\hat{\mu}$ for the notation of slope. 
Namely 
for 
$E \in D^b \Coh_{\le 2}(X)$, 
its Hilbert polynomial is written as
\begin{align*}
\chi(E\otimes f^{\ast}\lL_Y^{\otimes m})=
\alpha_{2, f^{\ast}\omega}(E)m^2/2 + 
\alpha_{1, f^{\ast}\omega}(E) m + \alpha_{0, f^{\ast}\omega}(E)
\end{align*}
with $\alpha_{i, f^{\ast}\omega}(E) \in \mathbb{Q}$. 
Moreover $\alpha_{2, f^{\ast}\omega}(E)$ is a positive 
integer if $E$ is a two dimensional sheaf outside a 
codimension two subset in $X$. 
We set
\begin{align*}
\hat{\mu}_{f^{\ast}\omega}(E) 
&\cneq \frac{\alpha_{1, f^{\ast}\omega}(E)}{\alpha_{2, f^{\ast}\omega}(E)} \in 
\mathbb{Q} \cup \{ \infty\}.
\end{align*}
Here we set $\hat{\mu}_{f^{\ast}\omega}(E)=\infty$
if the denominator is zero. 
By the Riemann-Roch theorem, 
$\hat{\mu}_{f^{\ast}\omega}(E)$ is written as
\begin{align}\label{muRR}
\hat{\mu}_{f^{\ast}\omega}(E)=
\frac{(\ch_2(E)+ c_1(X) \ch_1(E)/2)f^{\ast}\omega}{\ch_1(E) f^{\ast}\omega^2}. 
\end{align}
The slope function $\hat{\mu}_{f^{\ast}\omega}$
satisfies the week see-saw property on $\Coh_{\le 2}(X)$
and $\pPPer_{\le 2}(X/Y)$, i.e. 
if there is an exact sequence
$0 \to F \to E \to G \to 0$
in $\Coh_{\le 2}(X)$ or $\pPPer_{\le 2}(X/Y)$, 
we have either
\begin{align*}
&\hat{\mu}_{f^{\ast}\omega}(F) \ge \hat{\mu}_{f^{\ast}\omega}(E) \ge \hat{\mu}_{f^{\ast}\omega}(G) \mbox{ or } \\
&\hat{\mu}_{f^{\ast}\omega}(F) \le \hat{\mu}_{f^{\ast}\omega}(E) \le \hat{\mu}_{f^{\ast}\omega}(G). 
\end{align*}
Hence the slope function 
$\hat{\mu}_{f^{\ast}\omega}$ 
defines weak stability conditions on 
$\Coh_{\le 2}(X)$ and $\pPPer_{\le 2}(X/Y)$:
\begin{defi}\label{defi:mu}
An object $E \in \Coh_{\le 2}(X)$ (resp.~$\pPPer_{\le 2}(X/Y)$)
is $\hat{\mu}_{f^{\ast}\omega}$-(semi)stable if 
for any exact sequence $0 \to F \to E \to G \to 0$ 
in $\Coh_{\le 2}(X)$
(resp.~$\pPPer_{\le 2}(X/Y)$), we have the 
inequality
\begin{align*}
\hat{\mu}_{f^{\ast}\omega}(F)<(\le) \hat{\mu}_{f^{\ast}\omega}(G). 
\end{align*}
\end{defi}
\begin{rmk}
If $E \in \Coh_{\le 2}(X)$ is scheme theoretically 
supported on a smooth 
surface $S \subset X$ with 
$f^{\ast}\omega|_{S}$ ample, then 
$E$ is $\hat{\mu}_{f^{\ast}\omega}$-(semi)stable 
if and only if it is a 
torsion free $f^{\ast}\omega|_{S}$-slope 
(semi)stable sheaf on $S$ in the classical sense. 
\end{rmk}
\begin{rmk}
In general $f^{\ast}\omega$ is not ample on a support of 
a two dimensional sheaf, so we need to 
take a little care in dealing with some properties
of $\hat{\mu}_{f^{\ast}\omega}$-stability. 
The existence of Harder-Narasimhan 
filtrations follows from a standard argument (say, using the same argument of~\cite[Lemma~3.6]{Todext}). 
The boundedness of $\hat{\mu}_{f^{\ast}\omega}$-semistable 
objects will follow from Lemma~\ref{EinCoh}
and Proposition~\ref{prop:bound} below. 
\end{rmk}

For $\mu \in \mathbb{Q}$, let
$(\pT_{f^{\ast}\omega}^{\mu}, \pF^{\mu}_{f^{\ast}\omega})$
be the pair of subcategories in $\pPPer_{\le 2}(X/Y)$ given 
as follows: 
\begin{align*}
\pT_{f^{\ast}\omega}^{\mu}&\cneq \langle E \in \pPPer_{\le 2}(X/Y) : 
E \mbox{ is } \hat{\mu}_{f^{\ast}\omega} \mbox{-semistable with } 
\hat{\mu}_{f^{\ast}\omega}(E)>\mu \rangle_{\rm{ex}} \\
\pF_{f^{\ast}\omega}^{\mu}&\cneq \langle E \in \pPPer_{\le 2}(X/Y) : 
E \mbox{ is } \hat{\mu}_{f^{\ast}\omega} \mbox{-semistable with } 
\hat{\mu}_{f^{\ast}\omega}(E)\le \mu \rangle_{\rm{ex}}. 
\end{align*}
By the existence of Harder-Narasimhan filtrations in $\pPPer_{\le 2}(X/Y)$, 
the pair of subcategories $(\pT_{f^{\ast}\omega}^{\mu}, \pF_{f^{\ast}\omega}^{\mu})$
forms a torsion pair on $\pPPer_{\le 2}(X/Y)$. 
The associated (shifted) tilting is given by 
\begin{align*}
\pA^{\mu}_{f^{\ast}\omega} \cneq \langle \pF_{f^{\ast}\omega}^{\mu}, \pT_{f^{\ast}\omega}^{\mu}[-1] 
\rangle_{\rm{ex}} \subset D^b \Coh_{\le 2}(X). 
\end{align*}
By a general theory of tilting, 
$\pA^{\mu}_{f^{\ast}\omega}$ is the heart of a bounded t-structure on 
$D^b \Coh_{\le 2}(X)$. 
By the construction, any object
$E \in \pA^{\mu}_{f^{\ast}\omega}$ satisfies
the inequality
\begin{align*}
\alpha_{1, f^{\ast}\omega}(E)
-\mu \cdot \alpha_{2, f^{\ast}\omega}(E) \le 0.
\end{align*}
Hence the category
\begin{align*}
\pB^{\mu}_{f^{\ast}\omega} \cneq \{ E \in \pA^{\mu}_{f^{\ast}\omega} :
\alpha_{1, f^{\ast}\omega}(E)
-\mu \cdot \alpha_{2, f^{\ast}\omega}(E) = 0
 \} 
\end{align*}
is an abelian subcategory of $\pA^{\mu}_{f^{\ast}\omega}$. 
Note that $\pB^{\mu}_{f^{\ast}\omega}$ is written as 
\begin{align*}
\pB^{\mu}_{f^{\ast}\omega}= \left\langle F, \pPPer_0(X/Y)[-1] :
\begin{array}{c}
F\in \pPPer_{\le 2}(X/Y) 
\mbox{ is } \hat{\mu}_{f^{\ast}\omega} \mbox{-semistable } \\
\mbox{  with } \hat{\mu}_{f^{\ast}\omega}(F)=\mu \end{array}
\right\rangle_{\rm{ex}}. 
\end{align*}
Let $\Phi$ be the equivalence given by (\ref{Deq}). 
We have the following lemma: 
\begin{lem}\label{Phires}
The equivalence $\Phi$ restricts to an equivalence
$\Phi \colon \oB_{f^{\ast} \omega}^{\mu} \stackrel{\sim}{\to}
\iB_{f^{\dag \ast} \omega}^{\mu}$.  
\end{lem}
\begin{proof}
By Lemma~\ref{lem:linear} and (\ref{muRR}), we have 
$\hat{\mu}_{f^{\ast} \omega}(E)= \hat{\mu}_{f^{\dag \ast}\omega}(\Phi(E))$
for any object $E \in \oPPer_{\le 2}(X/Y)$. 
Therefore the result is obvious. 
\end{proof}

\subsection{Some properties of $\pB_{f^{\ast}\omega}^{\mu}$}
This subsection is devoted to showing 
some properties of the abelian category $\pB_{f^{\ast}\omega}^{\mu}$. 
We first prove that
it contains any $\hat{\mu}_{f^{\ast}\omega}$-semistable sheaf
with slope $\mu$.  
We define the category $\cC_{f^{\ast}\omega}^{\mu}$ to be
\begin{align}\label{C2}
\cC_{f^{\ast}\omega}^{\mu}
\cneq 
\{ E \in \Coh_{\le 2}(X) : E \mbox{ is } \hat{\mu}_{f^{\ast}\omega} \mbox{-semistable with } \hat{\mu}_{f^{\ast}\omega}(E)=\mu\}. 
\end{align}
We have the following lemma: 
\begin{lem}\label{EinCoh}
We have 
$\cC_{f^{\ast}\omega}^{\mu} \subset
\pB_{f^{\ast}\omega}^{\mu}$. 
\end{lem}
\begin{proof}
Let us take an object $E \in \cC_{f^{\ast}\omega}^{\mu}$. 
Since $\pPPer(X/Y)$ is a tilting of $\Coh(X)$
by~\cite{MVB}, we have the distinguished triangle
\begin{align}\label{DisTri}
\hH_{p}^0(E) \to E \to \hH_{p}^1(E)[-1]
\end{align}
where $\hH_p^i(E) \in \pPPer_{\le 2}(X/Y)$ is the 
$i$-th cohomology of $E$ with respect to the 
t-structure on $D^b \Coh_{\le 2}(X)$ with 
heart $\pPPer_{\le 2}(X/Y)$. 
Applying $\dR f_{\ast}$, we obtain the 
distinguished triangle in $D^b \Coh_{\le 2}(Y)$
\begin{align*}
\dR f_{\ast} \hH_p^0(E) \to \dR f_{\ast} E 
\to \dR f_{\ast} \hH_p^1(E)[-1]. 
\end{align*}
Since $\dR f_{\ast}$ takes 
$\pPPer_{\le 2}(X/Y)$ to $\Coh_{\le 2}(Y)$, 
we have $\dR f_{\ast} \hH_p^1(E) \cong R^1 f_{\ast}E$, 
and it is a zero dimensional sheaf. Hence
$\hH_p^1(E) \in \pPPer_0(X/Y)$, 
and it remains to show that $\hH_p^0(E)$ is a
$\hat{\mu}_{f^{\ast}\omega}$-semistable object in 
$\pPPer_{\le 2}(X/Y)$. 
It is enough to check that 
$\Hom(F, \hH_p^0(E))=0$ for any 
$F \in \pPPer_{\le 1}(X/Y)$. 
This follows from 
(\ref{DisTri}) and the distinguished triangle
\begin{align*}
\hH^{-1}(F)[1] \to F \to \hH^0(F)
\end{align*}
with $\hH^i(F) \in \Coh_{\le 1}(X)$, 
together with the fact that $E$ is a pure two dimensional sheaf. 
\end{proof}
We next show that 
any object in $\pB_{f^{\ast}\omega}^{\mu}$
admits a certain 
filtration, which plays an important role in the proof of 
Theorem~\ref{thm:intro}.
Let $(\pT, \pF)$ be the torsion pair 
of $\Coh_0(X/Y)$ as in (\ref{pTF}). 
We have the following proposition: 
\begin{prop}\label{prop:filt}
For any $E \in \pB_{f^{\ast}\omega}^{\mu}$, 
there exists a filtration
\begin{align}\label{filt:B}
0=E_0 \subset E_1 \subset E_2 \subset E_3=E
\end{align}
such that $F_i=E_{i}/E_{i-1}$ satisfy
$F_1 \in \pF$, $F_2 \in \cC_{f^{\ast}\omega}^{\mu}$
and $F_3 \in \pT[-1]$. 
\end{prop}
\begin{proof}
By (\ref{Per0}), we
 may assume that $E \notin \oPPer_0(X/Y)[-1]$. 
We have the exact sequence in $\pB_{f^{\ast}\omega}^{\mu}$
\begin{align*}
0 \to F \to E \to T[-1] \to 0
\end{align*}
such that $F \in \pPPer_{\le 2}(X/Y)$
is $\hat{\mu}_{f^{\ast}\omega}$-semistable with 
$\hat{\mu}_{f^{\ast}\omega}(F)=\mu$
and $T \in \pPPer_0(X/Y)$. 
By (\ref{Per0}), we also have
the exact sequence in $\pB_{f^{\ast}\omega}^{\mu}$
\begin{align*}
0 \to F' \to T[-1] \to T'[-1] \to 0
\end{align*}
with $F' \in \pF$ and $T'[-1] \in \pT[-1]$. 
Combining the above two exact sequences, 
we have the subobject $E_2 \subset E$ in $\pB_{f^{\ast}\omega}^{\mu}$
with $E/E_2 \in \pT[-1]$ which fits into the exact
sequence 
\begin{align}\label{FEF}
0 \to F \to E_2 \to F' \to 0
\end{align}
in $\pB_{f^{\ast}\omega}^{\mu}$. 
Note that the $\hat{\mu}_{f^{\ast}\omega}$-semistability of 
$F$ implies that $F \in \Coh_{\le 2}(X)$, hence
$E_2 \in \Coh_{\le 2}(X)$. 
We set $E_1 \subset E_2$ to be the maximal 
one dimensional subsheaf of $E_2$. 
Note that $E_2/E_1$ is a pure two dimensional sheaf, 
hence it is an object in $\cC_{f^{\ast}\omega}^{\mu}$. 
It is enough to show that $E_1 \in \pF$. 
Applying $\pPhi$
in (\ref{PhiA}) to (\ref{FEF}), 
we obtain the distinguished triangle in $D^b \Coh (\pAA_Y)$
\begin{align*}
\pPhi(F) \to \pPhi(E_2) \to \pPhi(F'). 
\end{align*}
Here $\pPhi(F) \in \Coh_{\le 2}(\pAA_Y)$
and $\pPhi(F') \in \Coh_0(\pAA_Y)[-1]$. 
In particular, we have
$\hH^0(\pPhi(E_2))=\pPhi(F)$, 
which is 
pure two dimensional by the $\hat{\mu}_{f^{\ast}\omega}$-stability of $F$. 
Because there is an injection 
$\hH^0 \pPhi(E_1) \hookrightarrow \hH^0 \pPhi(E_2)$
in $\Coh(\pAA_Y)$, we have 
$\hH^0 \pPhi(E_1)=0$, hence
$\pPhi(E_1) \in \Coh_0(\pAA_Y)[-1]$. 
This implies that $E_1 \in \pF$. 
\end{proof}
The filtration in the above proposition may be 
interpreted as 
a Harder-Narasimhan filtration with respect to a certain 
weak stability condition on $\pB_{f^{\ast}\omega}^{\mu}$ 
in~\cite{Tcurve1}. Indeed, we have the following lemma: 
\begin{lem}\label{lem:C}
If we set $\cC_1=\pF$, $\cC_2=\cC_{f^{\ast}\omega}^{\mu}$
and $\cC_3=\pT[-1]$, we have 
$\Hom(\cC_i, \cC_j)=0$
for $i<j$. In particular, the filtration 
(\ref{filt:B}) is unique up to an isomorphism. 
\end{lem}
\begin{proof}
The result is obvious from the definition of 
$\cC_i$. 
\end{proof}
Finally we show the boundedness of 
the set of objects in $\pB_{f^{\ast}\omega}^{\mu}$
with a fixed Mukai vector.  
We define $\pGamma_{f^{\ast}\omega}^{\mu}$
to be 
\begin{align*}
\pGamma_{f^{\ast}\omega}^{\mu}
\cneq \Imm \left( v(\ast) \colon 
\pB_{f^{\ast}\omega}^{\mu} \to \Gamma  \right). 
\end{align*}
\begin{prop}\label{prop:bound}
For any $v\in \pGamma_{f^{\ast}\omega}^{\mu}$, the 
set of objects 
$E \in \pB^{\mu}_{f^{\ast}\omega}$ with $v(E)=v$ is bounded. 
\end{prop}
\begin{proof}
We write $v=(P, \beta, n)$. 
If $P=0$, we have $E \in \pPPer_0(X/Y)[-1]$ and the 
result follows since $\pPPer_0(X/Y)$ is the extension 
closure of $\oO_x$ for $x \in X \setminus \Ex(f)$
and a finite number of sheaves up to shift supported on 
$\Ex(f)$ (cf.~\cite{MVB}). 
Hence we may assume that $P$ is a non-zero class of an effective 
divisor in $X$. 
For $E \in \pB^{\mu}_{f^{\ast}\omega}$
with $v(E)=v$, there exists an exact sequence in 
$\pB^{\mu}_{f^{\ast}\omega}$
\begin{align*}
0 \to F \to E \to T[-1] \to 0
\end{align*}
such that $F$ is a
$\hat{\mu}_{f^{\ast}\omega}$-semistable
 object in 
$\pPPer_{\le 2}(X/Y)$,
and $T$ is an object 
in $\pPPer_0(X/Y)$. 
Applying the equivalence 
$\pPhi$ in (\ref{PhiA})
and forgetting 
the $\pAA_Y$-module structures, we obtain 
the distinguished triangle in 
$D^b \Coh_{\le 2}(Y)$
\begin{align}\label{tri:pPhi}
\pPhi(F) \to \pPhi(E) \to \pPhi(T)[-1]. 
\end{align}
Here $\pPhi(T)$ is a zero dimensional sheaf 
on $Y$, and $\pPhi(F)$ is a sheaf on $Y$
which is pure two dimensional by the 
$\hat{\mu}_{f^{\ast}\omega}$-semistability of
$F$. 

For $M \in \Coh_{\le 2}(Y)$, 
its Hilbert polynomial is written as 
\begin{align*}
\chi(M\otimes \lL_Y^{\otimes m})=
\alpha_{2, \omega}(M)m^2/2 + \alpha_{1, \omega}(M)m + \alpha_{0, \omega}(M)
\end{align*}
for $\alpha_{i, \omega}(M) \in \mathbb{Q}$. 
It 
defines the 
$\hat{\mu}_{\omega}$-stability on 
$\Coh_{\le 2}(Y)$ by setting $\hat{\mu}_{\omega}(M)=
\alpha_{1, \omega}(M)/\alpha_{2, \omega}(M)$ 
as in Subsection~\ref{subsec:Tilt}. 
Let
 $\hat{\mu}_{\omega}^{\rm{max}}(M)$ be the 
maximal $\hat{\mu}_{\omega}$-slope among the 
Harder-Narasimhan factors of $M\in \Coh_{\le 2}(Y)$ with respect to the
$\hat{\mu}_{\omega}$-stability. 
For a fixed $v$, we claim that the set 
\begin{align}\label{babove}
\{ \hat{\mu}_{\omega}^{\rm{max}}(\pPhi(F)) :
E \in \pB_{f^{\ast}\omega}^{\mu},  
v(E)=v\} \subset \mathbb{Q}
\end{align}
 is bounded above. 
For simplicity, we prove the claim only for the case of
$p=0$. 
The case of $p=-1$ is similarly proved. 
We need to recall a construction of the vector bundle
$\oE$ on $X$
which gives an equivalence (\ref{PhiA}). 
Let $\lL_X$ be a globally generated ample line bundle on $X$. 
By replacing $\lL_Y$ with $\lL_Y^{\otimes k}$ for $k \gg 0$, 
we may assume that there is a surjection of sheaves
$(\lL_Y^{\vee})^{\oplus m} \twoheadrightarrow R^1 f_{\ast}\lL_X^{\vee}$
for some $m>0$. Taking the adjunction, 
we obtain the exact sequence of vector bundles
\begin{align*}
0 \to \lL_X^{\vee} \to \oE' \to f^{\ast} (\lL_Y^{\vee})^{\oplus m} \to 0. 
\end{align*}
Then $\oE$ is given by $\oO_X \oplus \oE'$. 
By the above construction of $\oE$, 
an upper bound of $\hat{\mu}_{\omega}^{\rm{max}}(\pPhi(F))$
is obtained if we give upper bounds of 
$\hat{\mu}_{\omega}^{\rm{max}}(\dR f_{\ast}F)$
and $\hat{\mu}_{\omega}^{\rm{max}}(f_{\ast}(F\otimes \lL_X))$, 
where $f_{\ast}(F\otimes \lL_X) \cneq \hH^0 \dR f_{\ast}(F\otimes \lL_X)$. 

Because $F\in \oPPer_{\le 2}(X/Y)$ 
is $\hat{\mu}_{f^{\ast}\omega}$-semistable, 
$\dR f_{\ast}F$ is a
$\hat{\mu}_{\omega}$-semistable
sheaf on $Y$. 
Hence $\hat{\mu}_{\omega}^{\rm{max}}(\dR f_{\ast}F)=\hat{\mu}_{\omega}(\dR f_{\ast}F)=\hat{\mu}_{f^{\ast}\omega}(E)$
which is constant. 
As for $f_{\ast}(F\otimes \lL_X)$, let
$F' \in \Coh_{\le 2}(Y)$ be the $\hat{\mu}_{\omega}$-semistable factor of 
$f_{\ast}(F\otimes \lL_X)$
such that $\hat{\mu}_{\omega}^{\rm{max}}(f_{\ast}(F\otimes \lL_X))
=\hat{\mu}_{\omega}(F')$. 
Since $\dR f_{\ast}F$ is $\hat{\mu}_{\omega}$-semistable, we have the 
inequality
\begin{align*}
\hat{\mu}_{\omega}(F' \otimes f_{\ast}\lL_X^{-1}) \le \hat{\mu}_{\omega}(\dR f_{\ast} F)
\end{align*}
which implies that
\begin{align*}
\hat{\mu}_{\omega}(F') \le \hat{\mu}_{f^{\ast}\omega}(E) + 
\frac{[F']f_{\ast} c_1(\lL_X) \omega}{[F']\omega^2}. 
\end{align*}
Here $[F'] \in H^2(Y)$ is the fundamental class of $F'$.
Since $[F']$ has only a finite number of possibilities, 
it follows that $\hat{\mu}_{\omega}(F')$ is bounded above. 

By the upper boundedness of (\ref{babove}),  
the result of Langer~\cite[Theorem~4.4]{Langer}, ~\cite[Theorem~3.8]{Langer2}
shows that $\alpha_{0, \omega}(\pPhi(F))$ is bounded above.  
Because $\alpha_{0, \omega}(\pPhi(T))$ is non-negative
and 
\begin{align*}
\alpha_{0, \omega}(\pPhi(F))- \alpha_{0, \omega}(\pPhi(T))=
\alpha_{0, \omega}(\pPhi(E))
\end{align*} is constant,
the set
\begin{align*}
\{(\alpha_{0, \omega}(\pPhi(F)), \alpha_{0, \omega}(\pPhi(T))) :
E \in \pB_{f^{\ast}\omega}^{\mu}, v(E)=v\}
\end{align*}
is a finite set. 
Again by~\cite[Theorem~4.4]{Langer}, ~\cite[Theorem~3.8]{Langer2},
and noting that $\pPhi(T)$ is a zero dimensional 
sheaf with bounded length, 
the set of objects
\begin{align}\label{setO}
\{ \pPhi(F), \pPhi(T) : E \in \pB_{f^{\ast}\omega}^{\mu}, v(E)=v\}
\end{align}
is bounded as $\oO_Y$-modules. 
Now for $M \in \Coh_{\le 2}(Y)$, 
the set of $\pAA_Y$-module structures on $M$
is contained in the set of morphisms
$\pAA_Y \to \eE nd(M)$, which is finite dimensional. 
Therefore the set of objects (\ref{setO})
is 
bounded also as $\pAA_Y$-modules. By the distinguished 
triangle (\ref{tri:pPhi}), 
the set of objects $\pPhi(E)$
for $E \in \pB_{f^{\ast}\omega}$ with $v(E)=v$
is also bounded as $\pAA_Y$-modules,
hence so is the set of such 
$E$ as $\pPhi$
gives an equivalence (\ref{PhiA}). 
\end{proof}
The following corollary is immediate from the 
above proposition: 
\begin{cor}\label{cor:finite}
For any $0\neq v\in \pGamma_{f^{\ast}\omega}^{\mu}$, 
there is only a finite number of ways to 
decompose $v$ into 
$v_1 + \cdots + v_l$ for some $l \ge 1$
and $0\neq v_i\in \pGamma_{f^{\ast}\omega}^{\mu}$. 
\end{cor}

\section{Flop formula of Donaldson-Thomas type invariants}\label{sec:DTflop}
This section is devoted to proving 
Theorem~\ref{thm:intro}. In this section, 
we always assume that $f \colon X \to Y$
is a 3-fold flopping contraction with $X, Y$ projective, 
and $\omega$ is an ample divisor on $Y$. 
We use the notation of the flop diagram (\ref{flop:dia}). 
\subsection{Hall algebras}\label{subsec:Hall}
Let $\mM$ be the moduli stack of objects 
$E \in D^b \Coh(X)$ with $\Ext^{<0}(E, E)=0$, 
which is an algebraic stack locally of finite type
(cf.~\cite{LIE}). 
The same arguments as in~\cite{Tst3}
easily imply that we have the open substack
\begin{align*}
\oO bj(\pB_{f^{\ast}\omega}^{\mu}) \subset \mM
\end{align*}
which parametrizes all the objects $E \in \pB_{f^{\ast}\omega}^{\mu}$. 
It decomposes into the connected components
\begin{align*}
\oO bj(\pB_{f^{\ast}\omega}^{\mu})
=\coprod_{v \in \pGamma_{f^{\ast}\omega}^{\mu}}
\oO bj_{v}(\pB_{f^{\ast}\omega}^{\mu})
\end{align*}
where $\oO bj_{v}(\pB_{f^{\ast}\omega}^{\mu})$
parametrizes objects $E \in \pB_{f^{\ast}\omega}^{\mu}$
with $v(E)=v$. 
By Lemma~\ref{prop:bound}, 
$\oO bj_{v}(\pB_{f^{\ast}\omega}^{\mu})$
is an algebraic stack of finite type. 

Recall that the stack theoretic
Hall algebra 
$H(\pB_{f^{\ast}\omega}^{\mu})$
of $\pB_{f^{\ast}\omega}^{\mu}$ is 
$\mathbb{Q}$-spanned by the isomorphism 
classes of the symbols (cf.~\cite[Section~2.2]{Joy4})
\begin{align*}
[\xX \stackrel{\rho}{\to} \oO bj(\pB_{f^{\ast}\omega}^{\mu})]
\end{align*}
where $\xX$ is an algebraic stack of finite type with 
affine geometric stabilizers and
$\rho$ is a 1-morphism. 
The relation is generated by 
\begin{align}\label{relation}
[\xX \stackrel{\rho}{\to} \oO bj(\pB_{f^{\ast}\omega}^{\mu})]
\sim [\yY \stackrel{ \rho|_{\yY} }{\to} \oO bj(\pB_{f^{\ast}\omega}^{\mu})]
+ [\uU \stackrel{\rho|_{\uU}}{\to} \oO bj(\pB_{f^{\ast}\omega}^{\mu})]
\end{align}
where $\yY \subset \xX$ is a closed substack and 
$\uU\cneq \xX \setminus \yY$. 
There is an associative $\ast$-product
on $H(\pB_{f^{\ast}\omega}^{\mu})$
based on the Ringel-Hall algebras. 
Let $\eE x(\pB_{f^{\ast}\omega}^{\mu})$ be the 
stack of short exact sequences $0 \to E_1 \to E_3 \to E_2 \to 0$
in $\pB_{f^{\ast}\omega}^{\mu}$
and $p_i \colon \eE x(\pB_{f^{\ast}\omega}^{\mu}) \to 
\oO bj(\pB_{f^{\ast}\omega}^{\mu})$ the 
1-morphism sending 
$E_{\bullet}$ to $E_i$. 
The $\ast$-product on $H(\pB_{f^{\ast}\omega}^{\mu})$ 
is given by 
\begin{align*}
[\xX_1 \stackrel{\rho_1}{\to} \oO bj(\pB_{f^{\ast}\omega}^{\mu})]
\ast [\xX_2 \stackrel{\rho_2}{\to} \oO bj(\pB_{f^{\ast}\omega}^{\mu})]
=[\xX_3 \stackrel{\rho_3}{\to} \oO bj(\pB_{f^{\ast}\omega}^{\mu})]
\end{align*}
where $(\xX_3, \rho_3=p_3 \circ (\rho_1', \rho_2'))$ is given by
the following Cartesian diagram
\begin{align*}
\xymatrix{
\xX_3 \ar[r]^{\hspace{-5mm}(\rho_1', \rho_2')}\ar[d] \ar@{}[dr]|\square
& \eE x(\pB_{f^{\ast}\omega}^{\mu}) \ar[d]^{(p_1, p_2)}
  \ar[r]^{p_3} &
\oO bj(\pB_{f^{\ast}\omega}^{\mu}) \\
\xX_1 \times \xX_2  \ar[r]^{\hspace{-5mm}(\rho_1, \rho_2)} 
& \oO bj(\pB_{f^{\ast}\omega}^{\mu})^{\times 2}.
& }
\end{align*}
The unit element is given by 
\begin{align*}
1 =[\Spec \mathbb{C} \to \oO bj(\pB_{f^{\ast}\omega}^{\mu})]\in H(\pB_{f^{\ast}\omega}^{\mu})
\end{align*}
which corresponds to $0 \in \bB_{f^{\ast}\omega}^{\mu}$. 
The algebra $H(\pB_{f^{\ast}\omega}^{\mu})$
is $\pGamma_{f^{\ast}\omega}^{\mu}$-graded: it decomposes as
\begin{align*}
H(\pB_{f^{\ast}\omega}^{\mu})
=\bigoplus_{v\in \pGamma_{f^{\ast}\omega}^{\mu}}
H_v(\pB_{f^{\ast}\omega}^{\mu})
\end{align*}
such that 
$H_v(\pB_{f^{\ast}\omega}^{\mu}) \ast H_{v'}(\pB_{f^{\ast}\omega}^{\mu})
\subset H_{v+v'}(\pB_{f^{\ast}\omega}^{\mu})$. 
The component $H_v(\pB_{f^{\ast}\omega}^{\mu})$
is $\mathbb{Q}$-spanned by the elements of the form 
$[\xX \stackrel{\rho}{\to} \oO bj_v(\pB_{f^{\ast}\omega}^{\mu})]$. 

We consider the following completion of 
$H(\pB_{f^{\ast}\omega}^{\mu})$
\begin{align*}
\widehat{H}(\pB_{f^{\ast}\omega}^{\mu})
\cneq \prod_{v\in \pGamma_{f^{\ast}\omega}^{\mu}}
H_v(\pB_{f^{\ast}\omega}^{\mu}).
\end{align*}
By Corollary~\ref{cor:finite}, 
the $\ast$-product on 
$H(\pB_{f^{\ast}\omega}^{\mu})$
extends to 
the $\ast$-product on 
$\widehat{H}(\pB_{f^{\ast}\omega}^{\mu})$. 
Moreover for any $\gamma \in \widehat{H}(\pB_{f^{\ast}\omega}^{\mu})$
with zero $H_0(\pB_{f^{\ast}\omega}^{\mu})$-component, 
the following elements are well-defined:
\begin{align*}
\exp(\gamma), \ 
\log(1+\gamma), \ (1+\gamma)^{-1} \in \widehat{H}(\pB_{f^{\ast}\omega}^{\mu}). 
\end{align*}
For a subcategory $\cC \subset \pB_{f^{\ast}\omega}^{\mu}$, 
suppose that there are constructible 
subsets $\oO bj_v(\cC) \subset \oO bj_v(\pB_{f^{\ast}\omega}^{\mu})$
whose closed points correspond to objects $E \in \cC$
with $v(E)=v$.  
By the relation (\ref{relation}), we are able to define the following element: 
\begin{align}\label{delta}
\delta_{\cC} \cneq \sum_{v \in \pGamma_{f^{\ast}\omega}^{\mu}}
[\oO bj_v(\cC) \subset \oO bj_v(\pB_{f^{\ast}\omega}^{\mu})] 
\in \widehat{H}(\pB_{f^{\ast}\omega}^{\mu}). 
\end{align}
By~\cite[Section~5.2]{Joy2}, 
there is a Lie subalgebra
$\widehat{H}^{\mathrm{Lie}}(\pB_{f^{\ast}\omega}^{\mu})$
of $\widehat{H}(\pB_{f^{\ast}\omega}^{\mu})$, 
consisting of elements called 
\textit{virtual indecomposable objects}. 
It contains elements of the 
form $[\xX \stackrel{\rho}{\to} \oO bj(\pB_{f^{\ast}\omega}^{\mu})]$
with $\xX$ a $\mathbb{C}^{\ast}$-gerbe over an algebraic space. 
Moreover, it also contains the elements of the form
\begin{align}\label{e-ele}
\epsilon_{\cC} \cneq \log \delta_{\cC} 
\in \widehat{H}^{\mathrm{Lie}}(\pB_{f^{\ast}\omega}^{\mu}). 
\end{align}

\subsection{Integration map}\label{subsec:int}
Let $\chi$ be the pairing on $\pGamma_{f^{\ast}\omega}^{\mu}$, 
given by 
\begin{align}\label{chi}
\chi(v_1, v_2)
=D_2 \beta_1 -D_1 \beta_2 -\frac{1}{2} c_1(X) D_1 D_2
\end{align}
where we write 
$v_i=(D_i, \beta_i, n_i)$. 
Note that $\chi$ is anti-symmetric 
if either $D_1$ or $D_2$ or $c_1(X)$ 
is zero. 
Let us take 
$E_i \in \pB_{f^{\ast}\omega}^{\mu}$ with 
$v(E_i)=v_i$. 
If $D_1$ or $D_2$ or $c_1(X)$ is zero, 
we have the equality
\begin{align}\label{chi12}
\chi(v_1, v_2)&=
\dim \Hom(E_1, E_2) - \dim \Ext^1(E_1, E_2) \\
\notag
&\hspace{20mm}+ \dim \Ext^1(E_2, E_1) - \dim \Hom(E_2, E_1)
\end{align}
by the Riemann-Roch theorem and the Serre duality. 

Let $\nu$ be a locally constructible function on 
$\mM$, satisfying the following assumption:  
\begin{assume}\label{assume}
The function $\nu$ is one of the following:  
\begin{itemize}
\item $\nu$ is Behrend 
function~\cite{Beh}
on the algebraic stack $\mM$ (cf.~\cite[Proposition~4.4]{JS}), 
denoted by $\chi_B$.  
In this case, we always assume that $X$ is a Calabi-Yau 3-fold,
i.e. $K_X=0$ and $H^1(\oO_X)=0$.  
\item $\nu \equiv 1$. 
In this case, $X$ is an arbitrary smooth projective 3-fold. 
\end{itemize}
\end{assume}
We define $\widehat{C}(\pGamma_{f^{\ast}\omega}^{\mu})$ 
to be
\begin{align*}
\widehat{C}(\pGamma_{f^{\ast}\omega}^{\mu})
\cneq \prod_{v \in \pGamma_{f^{\ast}\omega}^{\mu}}
\mathbb{Q} \cdot c_v
\end{align*}
with bracket $[c_{v_1}, c_{v_2}]_{\nu}$
given by
\begin{align}\label{brac}
[c_{v_1}, c_{v_2}]_{\nu} \cneq 
(-1)^{\epsilon(v_1, v_2)} \chi(v_1, v_2) c_{v_1 + v_2}.
\end{align}
Here $\epsilon(v_1, v_2)=\chi(v_1, v_2)$
if $\nu=\chi_B$, 
and $\epsilon(v_1, v_2)=0$ if $\nu \equiv 1$. 
Let $\pGamma_0 \subset \pGamma_{f^{\ast}\omega}^{\mu}$
be the subset of $(D, \beta, n)$ with $D=0$. 
Note that $\pGamma_0$ is the set of Mukai 
vectors of objects in $\pPPer_{0}(X/Y)[-1]$, and 
it is independent of $\mu$ and $\omega$. 
For three elements
$c_{v_1}, c_{v_2}, c_{v_3}$ in $\pGamma_{f^{\ast}\omega}^{\mu}$, 
the 
bracket (\ref{brac}) satisfies the Leibniz rule
if two of $v_i$ are contained in $\pGamma_0$, or 
$c_1(X)=0$. 
By Lemma~\ref{cor:finite}, the
bracket (\ref{brac}) defines the well-defined bracket on 
$\widehat{C}(\pGamma_{f^{\ast}\omega}^{\mu})$.  

The same arguments of~\cite[Theorem~5.12]{JS}, \cite[Theorem~6.12]{Joy2}
show that 
there is a $\mathbb{Q}$-linear 
homomorphism called \textit{integration map}
\begin{align}\notag
\Pi^{\nu} \colon 
\widehat{H}^{\mathrm{Lie}}(\pB_{f^{\ast}\omega}^{\mu})
\to \widehat{C}(\pGamma_{f^{\ast}\omega}^{\mu})
\end{align}
such that
if $\xX$ is a $\mathbb{C}^{\ast}$-gerbe 
over an algebraic space $\xX'$, 
we have 
\begin{align}\label{Pinu}
\Pi^{\nu}(
[\xX \stackrel{\rho}{\to} 
\oO bj_v(\pB_{f^{\ast}\omega}^{\mu})])
= \left(\sum_{k\in \mathbb{Z}}
k \cdot \chi(\nu^{-1}(k)) \right)c_{v}
\end{align} 
where $\chi(\ast)$ on the RHS is taken for 
the constructible subsets in $\xX'$. 
Let us take 
$\gamma_i \in \widehat{H}^{\mathrm{Lie}}(\pB_{f^{\ast}\omega}^{\mu}) \cap 
H_{v_i}(\pB_{f^{\ast}\omega}^{\mu})$
with $i=1, 2$, 
and suppose that one of $v_i$ is contained in $\pGamma_0$, 
or $c_1(X)=0$. 
Moreover if $\nu=\chi_B$, 
suppose that 
$\gamma_i$ is written as
\begin{align*}
\gamma_i=\sum_{j}a_{i, j} [ \xX_{i, j} \stackrel{\rho_{i, j}}{\to}
\oO bj_{v_i}(\pB_{f^{\ast}\omega}^{\mu})]
\end{align*}
with $a_{i, j} \in \mathbb{Q}$
such that for any closed point 
$x \in \xX_{i, j}$, the object in $\pB_{f^{\ast}\omega}^{\mu}$
corresponding the point  
$\rho_{i, j}(x)$ is a \textit{sheaf}. 
Then together with (\ref{chi12})
and Assumption~\ref{assume}, 
the arguments of~\cite[Theorem~5.12]{JS}
for $\nu=\chi_B$ and~\cite[Theorem~6.12]{Joy2}
for $\nu \equiv 1$ 
show that 
\begin{align}\label{liehom}
\Pi^{\nu}[\gamma_1, \gamma_2]_{\nu}= [\Pi^{\nu}(\gamma_1), 
\Pi^{\nu}(\gamma_2)]_{\nu}. 
\end{align}
\begin{rmk}
By~\cite[Theorem~5.5]{JS}, the 
condition that $\rho_{i, j}(x)$ corresponds to a sheaf
implies that the stack $\oO bj(\pB_{f^{\ast}\omega}^{\mu})$
is analytic locally at $\rho_{i, j}(x)$
written as a critical locus of a
holomorphic Chern Simons function.  
This property is required to show (\ref{liehom}) 
in the proof of~\cite[Theorem~5.12]{JS}. 
\end{rmk}
\subsection{Donaldson-Thomas type invariants}\label{subsec:DT}
Let $\cC_{f^{\ast}\omega}^{\mu} \subset \pB_{f^{\ast}\omega}^{\mu}$ be the 
subcategory in Lemma~\ref{EinCoh}, 
and consider the element (\ref{e-ele})
for $\cC=\cC_{f^{\ast}\omega}^{\mu}$. 
We define the DT type invariant 
which depends on a choice of $\nu$ in 
Assumption~\ref{assume}
as follows: 
\begin{defi}
For $v \in \pGamma_{f^{\ast}\omega}^{\mu}$, 
we define 
$\DT^{\nu}_{f^{\ast}\omega}(v) \in \mathbb{Q}$
by the following formula: 
\begin{align}\label{Pi:DT}
\Pi^{\nu}(\epsilon_{\cC_{f^{\ast}\omega}^{\mu}})
= \sum_{v \in \pGamma_{f^{\ast}\omega}^{\mu} }
(-1)^{\epsilon(\nu)} \DT^{\nu}_{f^{\ast}\omega}(v) \cdot c_v. 
\end{align} 
Here $\epsilon(\nu)=1$ if $\nu=\chi_B$
and $\epsilon(\nu)=0$ if $\nu\equiv 1$.
\end{defi}
\begin{rmk}
It is obvious that $\DT_{f^{\ast}\omega}^{\nu}(v)$ does not depend on 
a choice of $p$. 
Alternatively, the invariant $\DT_{f^{\ast}\omega}^{\nu}(v)$ 
can be defined by applying the 
integration map to the element $\epsilon_{\cC_{f^{\ast}\omega}^{\mu}}$
in the Hall algebra of $\Coh_{\le 2}(X)$, as in~\cite{JS}. 
When $\nu=\chi_B$, the invariant 
$\DT_{f^{\ast}\omega}^{\nu}(v)$ is the 
generalized DT invariant $\DT_{f^{\ast}\omega}(v)$ 
in Theorem~\ref{thm:intro}. 
\end{rmk}
\begin{rmk}\label{rmk:GIT}
Let $\mM_{f^{\ast}\omega}(v)$ be the 
moduli stack of 
$\hat{\mu}_{f^{\ast}\omega}$-semistable 
sheaves $E \in \Coh_{\le 2}(X)$ with 
$v(E)=v$. 
If any $[E] \in \mM_{f^{\ast}\omega}(v)$ is 
$\hat{\mu}_{f^{\ast}\omega}$-stable
then $\mM_{f^{\ast}\omega}(v)$ is the $\mathbb{C}^{\ast}$-gerbe
over a projective scheme $M_{f^{\ast}\omega}(v)$. 
Indeed, the stack $\mM_{f^{\ast}\omega}(v)$
coincides with the stack of $f^{\ast}\omega +\varepsilon H$-Gieseker
semistable sheaves 
with Mukai vector $v$
for $0<\varepsilon \ll 1$, consisting of only 
$f^{\ast}\omega +\varepsilon H$-Gieseker stable sheaves. 
Since $f^{\ast}\omega+\varepsilon H$ is ample for $0<\varepsilon \ll 1$, 
the coarse moduli space of $\mM_{f^{\ast}\omega}(v)$
is constructed by GIT quotient \emph{(cf.~\cite[Section~4]{HL})}. 
By (\ref{Pinu}), the invariant
$\DT_{f^{\ast}\omega}^{\nu}(\omega)$ coincides with 
the $\nu$-weighted Euler characteristic of $M_{f^{\ast}\omega}(v)$
up to sign. 
For instance, this is the case if $v=(P, \beta, n)$ and 
$P$ does not decompose into 
$P_1+P_2$ for effective divisor classes $P_i$.
\end{rmk}
For $n\in \mathbb{Z}$ and $\beta \in H_2(X/Y)$
with $\beta \ge 0$,  
there is another invariant
\begin{align}\label{invN}
N_{n, \beta}^{\nu} \in \mathbb{Q}
\end{align}
defined to be the DT type invariant 
counting $\mu_H$-semistable sheaves
$F \in \Coh_{0}(X/Y)$ satisfying
\begin{align*}
(\ch_2(F), \ch_3(F))=([F], \chi(F))=(\beta, n).
\end{align*} 
Here the 
$\mu_H$-stability on $\Coh_{0}(X/Y)$
is defined in Subsection~\ref{subsec:perv}, 
and $\nu$ is either $\nu=\chi_B$ or $\nu\equiv 1$
on the moduli stack of objects in $\Coh_0(X/Y)$
as in Assumption~\ref{assume}. 
We refer to~\cite[Subsection~6.4]{JS}, 
~\cite[Proposition-Definition~5.7]{Tcurve1},~\cite[Section~2.3]{Todmu}
for details.
In our situation, 
the invariant (\ref{invN}) for 
$n\le 0$ and $\beta>0$
can be defined in the following way. 
For $\mu' \in \mathbb{Q}_{\le 0}$, 
let $\cC_H^{\mu'} \subset \iF$
be the subcategory consisting of $\mu_H$-semistable 
sheaves $F$ with $\mu_H(F)=\mu'$ which makes sense
by Lemma~\ref{desc:F}.  
Then $N_{n, \beta}^{\nu}$ is defined by  
\begin{align}\label{def:N}
\Pi^{\nu}\epsilon_{\cC_H^{\mu'}}=
\sum_{n/H \cdot \beta =\mu'}
(-1)^{\epsilon(\nu)}N_{n, \beta}^{\nu} \cdot c_{(0, \beta, n)}. 
\end{align}
The invariant (\ref{invN}) is known to be independent of $H$
(cf.~\cite[Theorem~6.16]{JS}). 
By convention, we set $N_{n, \beta}^{\nu}=N_{-n, -\beta}^{\nu}$
for $\beta \le 0$. 
The invariant (\ref{invN})
 satisfies the following relation (cf.~\cite[Lemma~5.1]{Tcurve2})
\begin{align}\label{rel:N}
N_{n, \beta}^{\nu}=N_{-n, -\beta}^{\nu}=N_{-n, \beta}^{\nu}. 
\end{align}
We have the following lemma: 
\begin{lem}
We have the following identities: 
\begin{align}\label{eF}
&\Pi^{\nu}(\epsilon_{\oF})=
\sum_{n<0, \beta>0, f_{\ast}\beta=0} 
(-1)^{\epsilon(\nu)}N_{n, \beta}^{\nu} \cdot c_{(0, \beta, n)}
\\
\notag
&\Pi^{\nu}(\epsilon_{\iF})=
\sum_{n\le 0, \beta>0, f_{\ast}\beta=0} 
(-1)^{\epsilon(\nu)}N_{n, \beta}^{\nu} \cdot c_{(0, \beta, n)}.
\end{align}
\end{lem}
\begin{proof}
We prove the second identity.
By Lemma~\ref{desc:F}, 
the existence of Harder-Narasimhan filtrations with respect to the 
$\mu_H$-stability yields
\begin{align*}
\delta_{\iF}=\prod^{\longrightarrow}_{\mu'\le 0} \delta_{\cC_H^{\mu'}}. 
\end{align*} 
In the RHS, we take the product with the decreasing order 
of $\mu'$. We refer to~\cite[Theorem~5.11]{Joy4}
for details of the above 
product formula. 
We take the logarithm of both sides and apply 
the integration map. 
Since $\chi(v_1, v_2)=0$
for $v_i \in \iGamma_0$, 
the property of the integration map (\ref{liehom})
shows that
\begin{align*}
\Pi^{\nu} (\epsilon_{\iF})=\sum_{\mu' \le 0} 
\Pi^{\nu} (\epsilon_{\cC_H^{\mu'}}). 
\end{align*} 
By (\ref{def:N}), we obtain the desired identity. 
\end{proof}

\subsection{Flop formula}
In this subsection, we describe
the flop transformation formula of the invariants
$\DT_{f^{\ast}\omega}^{\nu}(v)$ in terms of 
the invariants $N_{n, \beta}^{\nu}$. 
\begin{prop}\label{prop:delta}
We have the following 
identity in $\widehat{H}(\pB_{f^{\ast}\omega}^{\mu})$: 
\begin{align}\label{form:delta}
\delta_{\pB_{f^{\ast}\omega}^{\mu}}= \delta_{\pF} \ast \delta_{\cC_{f^{\ast}\omega}^{\mu}} \ast \delta_{\pT[-1]}. 
\end{align}
\end{prop}
\begin{proof}
The equality (\ref{form:delta}) follows 
from Proposition~\ref{prop:filt}, Lemma~\ref{lem:C}
and the same argument of~\cite[Theorem~5.11]{Joy4}. 
\end{proof}
\begin{lem}
We have the following 
identity in $\widehat{H}(\pB_{f^{\ast}\omega}^{\mu})$:
\begin{align}\label{form:delta2}
\delta_{\pF} \ast \delta_{\cC_{f^{\ast}\omega}^{\mu}} \ast \delta_{\pF}^{-1} =
\delta_{\pB_{f^{\ast}\omega}^{\mu}} \ast \delta_{\pPPer_0(X/Y)[-1]}^{-1}. 
\end{align}
\end{lem}
\begin{proof}
The equality (\ref{form:delta2}) follows from 
(\ref{form:delta}) and 
$\delta_{\pF} \ast \delta_{\pT[-1]}=\delta_{\pPPer_0(X/Y)[-1]}$, 
where the last equality follows from (\ref{Per0}). 
\end{proof}
For $v\in \pGamma_{f^{\ast}\omega}^{\mu}$, we define
the invariant $\pDT_{f^{\ast}\omega}^{\nu}(v) \in \mathbb{Q}$
by 
\begin{align}\label{def:hat}
\Pi^{\nu}
\log \left( 
\delta_{\pB_{f^{\ast}\omega}^{\mu}} \ast \delta_{\pPPer_0(X/Y)[-1]}^{-1} \right)= \sum_{v \in \pGamma_{f^{\ast}\omega}^{\mu}}
\pDT_{f^{\ast}\omega}^{\nu}(v) \cdot c_v. 
\end{align}
For a fixed divisor class $P \in H^2(X)$, 
we set
\begin{align*}
&\DT_{f^{\ast}\omega}^{\nu, \mu}(P)
\cneq \sum_{(P, -\beta, -n) \in \pGamma_{f^{\ast}\omega}^{\mu}}
\DT_{f^{\ast}\omega}^{\nu}(P, -\beta, -n)q^n t^{\beta} \\
&\pDT_{f^{\ast}\omega}^{\nu, \mu}(P)
\cneq \sum_{(P, -\beta, -n) \in \pGamma_{f^{\ast}\omega}^{\mu}}
\pDT_{f^{\ast}\omega}^{\nu}(P, -\beta, -n)q^n t^{\beta}. 
\end{align*}
For a curve class $\beta$, 
we set 
$\langle P, \beta \rangle \in \mathbb{Z}$ to be
\begin{align}\label{def:delta}
\langle P, \beta \rangle \cneq \left\{ \begin{array}{ll}
(-1)^{P \cdot \beta -1} P \cdot \beta, & \nu =\chi_B \\
P \cdot \beta, & \nu \equiv 1.
\end{array} \right. 
\end{align}
\begin{prop}\label{prop:hat}
We have the following formulas:
\begin{align*}
\oDT_{f^{\ast}\omega}^{\nu, \mu}(P)
&=\prod_{\begin{subarray}{c}
n> 0, \beta>0 \\
f_{\ast}\beta=0
\end{subarray}}
\exp \left( N_{n, \beta}^{\nu}
q^n t^{-\beta} \right)^{\langle P, \beta \rangle} \cdot
\DT_{f^{\ast}\omega}^{\nu, \mu}(P) \\
\iDT_{f^{\ast}\omega}^{\nu, \mu}(P)
&=\prod_{\begin{subarray}{c}
n\ge 0, \beta>0 \\
f_{\ast}\beta=0
\end{subarray}}
\exp \left( N_{n, \beta}^{\nu}
q^n t^{-\beta} \right)^{\langle P, \beta \rangle} \cdot
\DT_{f^{\ast}\omega}^{\nu, \mu}(P). 
\end{align*}
\end{prop}
\begin{proof}
By the arguments so far, we have the following identities:
\begin{align*}
\sum_{v \in \pGamma_{f^{\ast}\omega}^{\mu}}
\pDT_{f^{\ast}\omega}^{\nu}(v) \cdot c_v &=
\Pi^{\nu}
\log \left( 
\delta_{\pB_{f^{\ast}\omega}^{\mu}} 
\ast \delta_{\pPPer_0(X/Y)[-1]}^{-1} \right) \\
&=
\Pi^{\nu} \log \left(\delta_{\pF} \ast \delta_{\cC_{f^{\ast}\omega}^{\mu}} \ast \delta_{\pF}^{-1}\right) \\
&=\Pi^{\nu}
\log \left( \exp \left(\epsilon_{\pF}\right) \ast 
\exp \left(\epsilon_{\cC_{f^{\ast}\omega}^{\mu}}\right) \ast
\exp \left(-\epsilon_{\pF}\right) \right) \\
&=\sum_{m \ge 0} 
\frac{1}{m!} \mathrm{Ad}^{\times m}_{\Pi^{\nu} \epsilon_{\pF}}
\left( \Pi^{\nu} \epsilon_{\cC_{f^{\ast}\omega}^{\mu}} \right). 
\end{align*}
Here the first equality is (\ref{def:hat}), 
the second equality is (\ref{form:delta2}), 
the third equality is (\ref{e-ele}), 
and the last equality follows from 
(\ref{liehom}) and the Baker-Campbell-Hausdorff formula. 
We note that we are able to apply (\ref{liehom})
since the objects in $\pF$ and $\cC_{f^{\ast}\omega}^{\mu}$ are sheaves. 
For instance, suppose that 
$p=0$ and $\nu=\chi_B$. 
Using (\ref{Pi:DT}) and (\ref{eF}), we obtain the 
following equality
\begin{align*}
\oDT_{f^{\ast}\omega}^{\nu}(P, \beta, n)=
&\sum_{m\ge 0} \frac{1}{m!}
\sum_{\begin{subarray}{c}
\beta_0 + \beta_1+ \cdots + \beta_m=\beta, \\
n_0 + n_1 + \cdots + n_m=n, \\
\beta_i >0, f_{\ast}\beta_i=0, n_i < 0, 1\le i\le m
\end{subarray}} \\
& 
\prod_{i=1}^m (-1)^{P \cdot \beta_i -1}(P \cdot \beta_i)
N_{n_i, \beta_i}^{\nu} \cdot \DT_{f^{\ast}\omega}^{\nu}(P, \beta_0, n_0). 
\end{align*}
Using (\ref{rel:N}), the desired formula 
follows from the above equality. The other cases are similarly discussed. 
\end{proof}
We define the following generating series: 
\begin{align}\label{def:DTgen}
\DT_{f^{\ast}\omega}^{\nu}(P) \cneq 
\sum_{\mu \in \mathbb{Q}} \DT_{f^{\ast}\omega}^{\nu, \mu}(P). 
\end{align}
Let $\psi_0$ and $\psi_1$ be the linear 
functions given in Lemma~\ref{lem:linear}. 
The following theorem is the main result of this subsection: 
\begin{thm}\label{thm:main}
We have the following formula:
\begin{align*}
\phi_{\ast} &\DT_{f^{\ast}\omega}^{\nu}(P) =q^{\psi_0(P)}t^{\psi_1(P)}
\DT_{f^{\dag \ast}\omega}^{\nu}(\phi_{\ast}P) \\
& \cdot \prod_{\begin{subarray}{c}
n> 0, \beta^{\dag}>0 \\
f^{\dag}_{\ast}\beta^{\dag}=0
\end{subarray}}
\exp 
\left( N_{n, \beta^{\dag}}^{\nu} q^n t^{\beta^{\dag}} 
\right)^{\langle \phi_{\ast}P, \beta^{\dag} \rangle}
\prod_{\begin{subarray}{c}
n\ge 0, \beta^{\dag}>0 \\
f^{\dag}_{\ast}\beta^{\dag}=0
\end{subarray}} 
\exp 
\left( N_{n, \beta^{\dag}}^{\nu} q^n 
t^{-\beta^{\dag}} \right)^{\langle \phi_{\ast}P, \beta^{\dag} \rangle}. 
\end{align*}
Here $\beta^{\dag}$ in the RHS are elements of $H_2(X^{\dag}/Y)$, 
and $\phi_{\ast}$ is the variable change 
$(n, \beta) \mapsto (n, \phi_{\ast}\beta)$. 
\end{thm}
\begin{proof}
By Lemma~\ref{Phires}, 
there is an isomorphism of algebras
\begin{align*}
\Phi_{\ast}^{H} \colon \widehat{H}(\oB_{f^{\ast}\omega}^{\mu})
\stackrel{\cong}{\to}
\widehat{H}(\iB_{f^{\dag \ast}\omega}^{\mu}). 
\end{align*}
By the commutative diagram (\ref{com:dia}), 
we have the commutative diagram
\begin{align*}
\xymatrix{
\widehat{H}^{\rm{Lie}}(\oB_{f^{\ast}\omega}^{\mu}) \ar[r]^{\Phi_{\ast}^H}
 \ar[d]_{\Pi^{\nu}} & 
\widehat{H}^{\rm{Lie}}(\iB_{f^{\dag \ast}\omega}^{\mu}) \ar[d]^{\Pi^{\nu}} \\
\widehat{C}(\oGamma_{f^{\ast}\omega}^{\mu}) \ar[r]^{\Phi_{\ast}^C} & 
\widehat{C}(\iGamma_{f^{\dag \ast}\omega}^{\mu}).
}
\end{align*}
Here $\Phi_{\ast}^C$ sends
$c_{v}$ to $c_{\Phi_{\ast}v}$. 
Moreover we have 
\begin{align*}
\Phi_{\ast}^H \left(\delta_{\oB_{f^{\ast}\omega}^{\mu}}\right)
=\delta_{\iB_{f^{\dag \ast}\omega}^{\mu}}, \ 
\Phi_{\ast}^H \left(\delta_{\oPPer_0(X/Y)[-1]}\right)=
\delta_{\iPPer_0(X^{\dag}/Y)[-1]}.
\end{align*}
Therefore we obtain the equality
\begin{align*}
\Phi_{\ast}^C \Pi^{\nu} \log 
\left( \delta_{\oB_{f^{\ast}\omega}^{\mu}} \ast 
 \delta_{\oPPer_0(X/Y)[-1]}^{-1} \right)
=\Pi^{\nu} \log 
\left( \delta_{\iB_{f^{\dag \ast}\omega}^{\mu}} \ast 
 \delta_{\iPPer_0(X^{\dag}/Y)[-1]}^{-1} \right).
\end{align*}
Using Lemma~\ref{lem:linear}, 
the above equality implies
\begin{align}\label{equal:hat}
\phi_{\ast}\oDT_{f^{\ast}\omega}^{\nu, \mu}(P)
=q^{\psi_0(P)} t^{\psi_1(P)} \cdot
\iDT_{f^{\dag \ast}\omega}^{\nu, \mu}(\phi_{\ast}P). 
\end{align}
The desired formula follows from the above equality
and Proposition~\ref{prop:hat}, 
together noting that 
\begin{align*}
\phi_{\ast}\prod_{\begin{subarray}{c}
n> 0, \beta>0 \\
f_{\ast}\beta=0
\end{subarray}}
\exp \left( N_{n, \beta}^{\nu}
q^n t^{-\beta} \right)^{\langle P, \beta \rangle}
=\prod_{\begin{subarray}{c}
n> 0, \beta^{\dag}>0\\
f_{\ast}^{\dag}\beta^{\dag}=0
\end{subarray}}
\exp \left( N_{n, \beta^{\dag}}^{\nu}
q^n t^{\beta^{\dag}} \right)^{-\langle \phi_{\ast}P, \beta^{\dag} \rangle}. 
\end{align*}
Here we have set $\beta^{\dag}=-\phi_{\ast}\beta \in H_2(X^{\dag}/Y)$, 
which is 
effective by Lemma~\ref{lem:cup}. 
We have also used the fact 
$N_{n, \beta}^{\nu}=N_{n, \beta^{\dag}}^{\nu}$ 
by~\cite[Theorem~5.6]{Tcurve2}, 
and $\langle P, \beta \rangle= -\langle \phi_{\ast}P, \beta^{\dag} \rangle$
by Lemma~\ref{lem:cup}. 
\end{proof}

\subsection{Parabolic stable pairs}
The error term of the formula in Theorem~\ref{thm:main}
is described in terms of invariants counting 
parabolic stable pairs, introduced in~\cite{Todpara}. 
Let $H$ be an $f$-ample effective divisor in $X$, 
which intersects with $\Ex(f)$ transversally. 
Recall that there is the $\mu_H$-stability on 
$\Coh_{0}(X/Y)$ given in Subsection~\ref{subsec:perv}. 
\begin{defi}
An $H$-parabolic stable pair consists of 
$(F, s)$,
where $F\in \Coh_{0}(X/Y)$ and 
$s \in F \otimes \oO_H$, 
 satisfying the following conditions:
\begin{itemize}
\item 
$F$ is a one dimensional $\mu_H$-semistable sheaf. 
\item  
For any surjection $\pi \colon 
F \twoheadrightarrow F'$
with $\mu_{H}(F)=\mu_H(F')$, we have 
$(\pi \otimes \oO_H)(s) \neq 0$. 
\end{itemize}
\end{defi}
For $\beta\in H_2(X/Y)$ and $n \in \mathbb{Z}$, let 
$M_{H}^{\rm{par}}(\beta, n)$ be the moduli 
space of parabolic stable pairs $(F, s)$
with $[F]=\beta$ and $\chi(F)=n$.
The 
moduli space $M_{H}^{\rm{par}}(\beta, n)$
is a projective scheme, 
which represents the functor 
of families of above parabolic stable pairs (cf.~\cite[Theorem~1.2]{Todpara}).
\begin{rmk}
In~\cite{Todpara}, we assumed the ampleness of $H$. 
However 
since any $F \in \Coh_{\le 1}(X)$ with 
$[F]=\beta \in H_2(X/Y)$ is an object in $\Coh_0(X/Y)$, 
the $f$-ampleness of $H$ is enough to prove
the existence of $M_{H}^{\rm{par}}(\beta, n)$. 
\end{rmk}
As in Assumption~\ref{assume}, let $\nu$ be a constructible function on
$M_H^{\rm{par}}(\beta, n)$ which is either 
$\nu=\chi_B$ or $\nu \equiv 1$.
The associated invariant of parabolic stable pairs
 is defined similarly to the
DT theory: 
\begin{defi}
For $\beta \in H_2(X/Y)$ and $n \in \mathbb{Z}$, we 
define the invariant 
$\mathrm{Par}_H^{\nu}(\beta, n) \in \mathbb{Z}$
to be 
\begin{align*}
\mathrm{Par}_H^{\nu}(\beta, n) \cneq 
\sum_{k \in \mathbb{Z}}k \cdot \chi(\nu^{-1}(k)).
\end{align*}
\end{defi}
We note that, contrary to the invariants $N_{n, \beta}^{\nu}$, 
the invariants $\mathrm{Par}_{H}^{\nu}(\beta, n)$ are always 
integers. However they are related as follows: 
for $\mu \in \mathbb{Q}$, 
let us set 
\begin{align*}
\mathrm{Par}_H^{\nu, \mu}(q, t) \cneq 1+
\sum_{\begin{subarray}{c}\frac{n}{\beta \cdot H} =\mu \\
f_{\ast} \beta=0, \beta>0 
\end{subarray}} \mathrm{Par}_H^{\nu}(\beta, n)q^n t^{\beta}. 
\end{align*}
Then we have the following identity (cf.~\cite[Theorem~4.12]{Todpara})
\begin{align}\label{Par=N}
\mathrm{Par}_H^{\nu, \mu}(q, t)
=\prod_{\begin{subarray}{c}
\frac{n}{\beta \cdot H} =\mu \\
f_{\ast} \beta=0, \beta>0 \end{subarray}}
\exp \left( N_{n, \beta}^{\nu} q^n t^{\beta}
  \right)^{\langle H, \beta \rangle}. 
\end{align}
For a divisor class $P$ on $X$, 
it is either 
$f$-ample, $f$-trivial, or $f$-anti ample
by the definition of a flopping contraction.  
We set $s=1$ in the first case, $s=0$ in the second case
and $s=-1$ in the last case. 
By the $f$-relative 
base point free theorem, 
there is an $f$-ample effective divisor $H$ on $X$
which intersects with $\Ex(f)$ transversally 
such that $P \cdot \beta=sH \cdot \beta$
for any $\beta\in H_2(X/Y)$. 
This implies that
\begin{align*}
\{\mathrm{Par}_H^{\nu, \mu}(q, t)\}^{s}
=\prod_{\begin{subarray}{c}
\frac{n}{\beta \cdot H} =\mu \\
f_{\ast} \beta=0, \beta>0 \end{subarray}}
\exp \left( N_{n, \beta}^{\nu} q^n t^{\beta}
  \right)^{\langle P, \beta \rangle}. 
\end{align*}
Therefore we obtain the following corollary
of Theorem~\ref{thm:main}:
\begin{cor}\label{cor:para}
In the same situation of Theorem~\ref{thm:main}, 
there is an $f^{\dag}$-ample effective divisor $H^{\dag}$
on $X^{\dag}$ which intersects with $\Ex(f^{\dag})$ transversally, 
and $s \in \{1, 0, -1\}$
such that the following identity holds: 
\begin{align*}
\phi_{\ast} \DT_{f^{\ast}\omega}^{\nu}(P) &=q^{\psi_0(P)}t^{\psi_1(P)}
\DT_{f^{\dag \ast}\omega}^{\nu}(\phi_{\ast}P) \\
&\hspace{20mm} \cdot \left\{\prod_{\mu \in \mathbb{Q}_{>0}} 
\mathrm{Par}_{H^{\dag}}^{\nu, \mu}(q, t)
\prod_{\mu\in \mathbb{Q}_{\ge 0}} 
\mathrm{Par}_{H^{\dag}}^{\nu, \mu}(q, t^{-1}) \right\}^{s}. 
\end{align*}
Here $s$ 
is determined by $\phi_{\ast}P \cdot \beta^{\dag}=sH^{\dag} \cdot \beta^{\dag}$
for any $\beta^{\dag} \in H_2(X^{\dag}/Y)$. 
\end{cor}
\subsection{Computation of the error term}
From this subsection until the end of this section, 
we assume that
$f$ contracts only an irreducible 
rational curve $C \subset X$ to 
a point $p\in Y$. 
We compute the error term 
of the formula
in Corollary~\ref{cor:para} 
under the above assumption. 
Let $H$ be an effective $f$-ample divisor on $X$
which intersects with $C$ transversally. 
As in Subsection~\ref{subsec:flop}, 
we denote by $l$ the scheme theoretic length of $f^{-1}(p)$ at the 
generic point of $C$, and 
$n_{j}$ for $j \in \mathbb{Z}_{\ge 1}$ the integers
in Definition~\ref{def:nj}.  
\begin{lem}\label{lem:com:N}
If $\nu=\chi_B$, we have the following formula
\footnote{We warn the readers not to confuse the 
use of `$n$' in the formula.
The `$n$' without subscript stands for
the third Chern character of a sheaf, while 
$`n_{m/k}'$ stands for 
the integers in Definition~\ref{def:nj}.} for $m\ge 1$:
\begin{align}\notag
N_{n, m[C]}^{\nu}
=\sum_{k\ge 1, k|(n, m)}\frac{n_{m/k}}{k^2}. 
\end{align}
\end{lem}
\begin{proof}
The 
invariant $N_{n, m[C]}^{\nu}$ for $m>0$ depends only on 
the formal completion (\ref{complet}) 
as any object which contributes to $N_{n, m[C]}^{\nu}$ is supported on $C$. 
The moduli space of Joyce-Song~\cite{JS} 
stable pairs $(\oO_{\widehat{X}} \to F(aH))$
for $a\gg 0$ on $\widehat{X}$
 is a projective scheme if 
$[F]=m[C]$, 
since the latter condition implies that $F$ is supported on $C$. 
Hence the same argument of~\cite{JS} shows that 
 $N_{n, m[C]}^{\nu}$ is invariant 
under the deformation (\ref{deform}) of $\widehat{X}$. 
Let $C_{j, k} \subset \widehat{\xX}_t$ be 
$(-1, -1)$-curves as in Subsection~\ref{subsec:flop}. 
The invariance of $N_{n, \beta}^{\nu}$ under the deformation (\ref{deform})
yields
\begin{align*}
N_{n, m[C]}^{\nu}
=\sum_{\begin{subarray}{c}
\sum_{1\le j\le l, 1\le k\le n_j} j m_{j, k}=m 
\end{subarray}}N_{n, \sum_{1\le j\le l, 1\le k\le n_j}
 m_{j, k}[C_{j, k}]}^{\nu}.
\end{align*}
The invariant
$N_{n, \sum_{1\le j\le l, 1\le k\le n_j}
 m_{j, k}[C_{j, k}]}^{\nu}$
in the RHS is non-zero
only if $m_{j, k} \neq 0$ for only one 
$(j, k)$ since $C_{j, k}$ are disjoint, 
and we have $j|m$, $m_{j, k}=m/j$ for such 
$(j, k)$. 
Since $C_{j, k}$ is a $(-1, -1)$-curve, 
the computation in~\cite[Example~6.2]{JS}
shows that
$N_{n, m_{j, k}[C_{j, k}]}^{\nu}$
is non-zero only if 
$m_{j, k}|n$, and it equals to 
$1/m_{j, k}^2=j^2/m^2$ in this case. 
Therefore we obtain 
\begin{align*}
N_{n, m[C]}^{\nu}=
\sum_{j\ge 1, j|m, (m/j)|n}
n_j
\frac{j^2}{m^2}. 
\end{align*}
 By putting $k=m/j$, we obtain the desired identity.
\end{proof}
\begin{prop}\label{prop:comp}
If $\nu=\chi_B$, 
we have the following formula:
\begin{align*}
\mathrm{Par}_{H}^{\nu, \mu}
(q, t)=
\prod_{\begin{subarray}{c}
1\le j\le l, \\ \frac{n}{j(H \cdot C)}=\mu
\end{subarray}}
\left(1-(-1)^{jH \cdot C} q^n t^{jC} \right)^{jn_j(H\cdot C)}. 
\end{align*}
\end{prop} 
\begin{proof}
By~\cite[Proposition~4.5]{Todpara}, the
 invariants $N_{n, \beta}^{\nu}$ 
satisfy the
following multiple cover 
formula
\begin{align}\label{mult}
N_{n, \beta}^{\nu}=\sum_{k\in \mathbb{Z}_{\ge 1}, k|(n, \beta)}
\frac{1}{k^2} N_{1, \beta/k}^{\nu}
\end{align}
if and only if we have the following formula:
\begin{align}\label{form:Par}
\mathrm{Par}_{H}^{\nu, \mu}
(q, t)=
\prod_{\begin{subarray}{c}
j\ge 1, \\ \frac{n}{j(H \cdot C)}=\mu
\end{subarray}}
\left(1-(-1)^{j(H \cdot C)} q^n t^{jC}
\right)^{j N_{1, j[C]}^{\nu}(H \cdot C)}. 
\end{align}
By Lemma~\ref{lem:com:N}, 
$N_{1, j[C]}^{\nu}=n_{j}$
(which also holds from~\cite{Katz}), and 
the identity (\ref{mult}) holds 
for any $\beta=j[C]$ with $j>0$. 
Therefore we obtain the desired formula. 
\end{proof}
\subsection{Proof of Theorem~\ref{thm:intro}}\label{subsec:proof}
The following theorem, which proves Theorem~\ref{thm:intro}, is the 
main result in this section: 
\begin{thm}\label{thm:intro2}
Suppose that $\nu=\chi_B$. 
Then  
we have the following formula: 
\begin{align*}
&\DT_{f^{\dag \ast} \omega}^{\nu}(\phi_{\ast}P)
=\phi_{\ast}\DT_{f^{\ast}\omega}^{\nu}(P)
 \\
&\hspace{20mm}\cdot 
\prod_{j=1}^{l} \left\{ 
i^{jP \cdot C-1}
\eta(q)^{-1}\vartheta_{1, 1}(q, ((-1)^{\phi_{\ast}P}t)^{jC^{\dag}}) \right\}^{jn_j P \cdot C}. 
\end{align*}
\end{thm}
\begin{proof}
By Proposition~\ref{lem:psi},  
Corollary~\ref{cor:para}
and Proposition~\ref{prop:comp}, we have
\begin{align*}
&\phi_{\ast} \DT_{f^{\ast}\omega}^{\nu}(P) 
=q^{-\sum_{j=1}^{l} jn_j(P\cdot C)/12}t^{-\sum_{j=1}^{l}j^2 n_j(P \cdot C)C^{\dag}/2}
\DT_{f^{\dag \ast}\omega}^{\nu}(\phi_{\ast}P) \\
& \cdot \prod_{j=1}^{l}
\left\{ \prod_{n>0}
 (1-(-1)^{j(H^{\dag} \cdot C^{\dag})} q^n t^{jC^{\dag}})
 \prod_{n\ge 0}
 (1-(-1)^{j(H^{\dag} \cdot C^{\dag})} q^n t^{-jC^{\dag}})
 \right\}^{sjn_j(H^{\dag} \cdot C^{\dag})}.
\end{align*}
The result of Theorem~\ref{thm:intro} 
follows from the above identity
together with 
$\phi_{\ast}P \cdot C^{\dag}=sH^{\dag} \cdot C^{\dag}=-P \cdot C$
by Lemma~\ref{lem:cup}, and 
the Jacobi triple product 
\begin{align}\label{theta:pro}
\vartheta_{1, 1}(q, t)=iq^{\frac{1}{12}}t^{\frac{1}{2}}
\eta(q) (1-t^{-1}) \prod_{n\ge 1}(1-q^n t) (1-q^n t^{-1}). 
\end{align}
\end{proof}
\begin{rmk}
If $l=1$, then $n_1=n$ where $n$ is the width of $C$, 
and we obtain the following formula:
\begin{align*}
\DT_{f^{\dag \ast}\omega}^{\nu}(\phi_{\ast}P) =
\phi_{\ast}\DT_{f^{\ast}\omega}^{\nu}(P) 
 \cdot \left\{i^{P \cdot C-1}\eta(q)^{-1}   
\vartheta_{1, 1}(q, ((-1)^{\phi_{\ast}P}t)^{C^{\dag}}) \right\}^{nP \cdot C}.
\end{align*}
\end{rmk}
\begin{rmk}\label{rmk:Mukai}
If we take the generating series (\ref{def:DTgen}) w.r.t. the Chern characters 
(not Mukai vectors), then the error term in Theorem~\ref{thm:intro2}
is further multiplied by some monomial in $q$, $t$, which is not
a Jacobi form. This indicates that taking the generating series 
w.r.t. the Mukai vectors is crucial in order to obtain the 
modularity of the generating series. 
\end{rmk}
\subsection{Euler characteristic version}
We next treat the case $\nu \equiv 1$. 
In this case, the invariant $N_{n, \beta}^{\nu}$ is not 
deformation invariant and its computation is more subtle. 
We focus on the case of $l=1$, and 
let $n$ be the width of $C$. 
We take a smooth divisor $\widehat{H}$ in 
$\widehat{X}$ 
such that $\widehat{H} \cap C$ consists of one point, 
which always exist. 
Also we set the polynomial $f_n(x)$ to be
\begin{align}\label{fn}
f_n(x) \cneq 1+ x+ \cdots + x^n. 
\end{align}
We have the following lemma: 
\begin{lem}\label{lem:compv}
In the above situation, we have 
$\mathrm{Par}_{\widehat{H}}^{\nu, \mu}(q, t)=1$ unless 
$\mu \in \mathbb{Z}$. 
If $\mu \in \mathbb{Z}$, we have
\begin{align*}
\mathrm{Par}_{\widehat{H}}^{\nu, \mu}(q, t)=f_{n}(q^{\mu}t^{C^{\dag}}). 
\end{align*}
\end{lem}
\begin{proof}
Let $F \in \Coh_{0}(X/Y)$ be 
a $\mu_{\widehat{H}}$-semistable sheaf with 
$[F]=m[C]$ for $m\ge 1$, 
and $F_1, \cdots, F_e$
the $\mu_{\widehat{H}}$-stable factors of 
$F$. 
Then each $F_i$ is $\oO_{f^{-1}(p)} =\oO_C$-module, 
hence isomorphic to $\oO_C(k_i)$ for some $k_i \in \mathbb{Z}$. 
Since $\mu_{\widehat{H}}(F)=\mu_{\widehat{H}}(F_i)=k_i +1$, the integer 
$k_i$ is independent of $i$,
and $\mu_{\widehat{H}}(F)$ is an integer. 
This implies that $\mathrm{Par}_{\widehat{H}}^{\nu, \mu}(q, t)=1$ unless
$\mu \in \mathbb{Z}$. 

Suppose that $\mu \in \mathbb{Z}$. 
It is enough to show that 
$\mathrm{Par}_{\widehat{H}}^{\nu}(mC, m\mu)$ equals to one 
for $1\le m\le n$, and zero otherwise. 
Applying $\otimes \oO_{\widehat{X}}(-\mu \widehat{H})$, 
we have $\mathrm{Par}_{\widehat{H}}^{\nu}(mC, m\mu)
=\mathrm{Par}_{\widehat{H}}^{\nu}(mC, 0)$, so we may 
assume that $\mu=0$. 
Let $(F, s)$ be a $\widehat{H}$-parabolic stable pair
on $\widehat{X}$ with 
$([F], \chi(F))=(m[C], 0)$. 
Then by the above argument, we have 
$F \in \langle \oO_C(-1) \rangle_{\rm{ex}}$. 
We first show that $F$ must be indecomposable. 
Indeed suppose that 
$F$ decomposes as $F_1 \oplus F_2$
with $F_i \neq 0$, and write 
$s=(s_1, s_2)$ with $s_i \in F_i \otimes \oO_{\widehat{H}}$. 
We choose surjections $\pi_i \colon F_i \twoheadrightarrow \oO_C(-1)$, 
and set $s_i'=\pi_i(s_i) \in \oO_C(-1) \otimes \oO_{\widehat{H}}
 \cong \mathbb{C}$. 
By the parabolic stability, 
we have $(s_1', s_2') \in \mathbb{C}^2 \setminus \{0\}$.  
Then the composition
\begin{align*}
\pi \colon F_1 \oplus F_2 \stackrel{(\pi_1, \pi_2)}{\twoheadrightarrow}
\oO_C(-1)^{\oplus 2} \stackrel{(s_2', -s_1')}{\twoheadrightarrow} \oO_C(-1)
\end{align*}
satisfies $(\pi \otimes \oO_{\widehat{H}})(s)=0$, which 
contradicts to the parabolic stability of $(F, s)$. 
Hence $F$ is indecomposable. 

Next we classify indecomposable objects in 
$\langle \oO_C(-1) \rangle_{\rm{ex}}$. 
The derived category $D^b \Coh(\widehat{X})$ is known 
to be equivalent to the derived category of 
finitely generated right modules over the 
completion of the path algebra $A$ of the quiver 
of the form (cf.~\cite[Example~3.10]{WM})
\begin{align}\label{quiver}
\xymatrix{
\circ \ar@(ul, dl)[]|{y_2} \ar@/^/[rr]|{a_2} \ar@/^1.5pc/[rr]|{a_1}&&
\bullet \ar@/^/[ll]|{b_1} \ar@/^1.5pc/[ll]|{b_2} \ar@(ur, dr)[]|{y_1}
} 
\end{align}
with relations given by $y_1 a_i = a_i y_2$, 
$y_2 b_i = b_i y_1$, 
$2y_1^n =a_1 b_1-a_2 b_2$ and 
$2y_2^n =b_1 a_1 - b_2 a_2$. 
Under the above derived equivalence, the category 
$\langle \oO_{C}(-1) \rangle_{\rm{ex}}$ is equivalent 
to the subcategory of $\modu A$ generated by a simple object
corresponding to one of the vertex in (\ref{quiver}), say $\bullet$. 
Hence $\langle \oO_C(-1) \rangle_{\rm{ex}}$ is equivalent to 
$\modu\left(\mathbb{C}[y_1]/(y_1^n)\right)$. 
The indecomposable objects in  the latter 
category consist of $\mathbb{C}[y_1]/(y_1^m)$ with 
$1\le m\le n$, hence $\mathrm{Par}_{\widehat{H}}^{\nu}(mC, 0)=0$
for $m>n$. 

Suppose that $1\le m \le n$, and 
$F \in \langle \oO_C(-1) \rangle_{\rm{ex}}$ corresponds to 
$\mathbb{C}[y_1]/(y_1^m)$. 
Let $\pi' \colon F\twoheadrightarrow F'$ be the quotient 
corresponding to the surjection
$\mathbb{C}[y_1]/(y_1^m) \twoheadrightarrow \mathbb{C}$. 
Then $\pi'$ factors through 
any quotient $F \twoheadrightarrow F''$
in $\langle \oO_C(-1) \rangle_{\rm{ex}}$, hence $(F, s)$
is a parabolic stable pair if and only if 
$0\neq (\pi' \otimes \oO_{\widehat{H}})(s) \in \mathbb{C}$. 
Since $F\otimes \oO_{\widehat{H}} \cong \mathbb{C}^m$, 
the set of such parabolic stable pairs 
is identified with $\mathbb{C}^{\ast} \times 
\mathbb{C}^{m-1}$. 
On the other hand, $\Aut(F)$ is isomorphic to 
$\Aut(\mathbb{C}[y_1]/(y_1^m))$ in $\modu(\mathbb{C}[y_1]/(y_1^n))$, 
and the latter group consists of 
$1 \mapsto a_0 + a_1 y_1 + \cdots + a_{m-1} y_1^{m-1}$
with $a_i \in \mathbb{C}$, $a_0 \neq 0$. 
Hence $\Aut(F)$ is isomorphic to $\mathbb{C}^{\ast} \times \mathbb{C}^{m-1}$. 
The element 
 $(a_0, a_1, \cdots, a_{m-1}) \in \Aut(F)$
acts on the 
above set $\mathbb{C}^{\ast} \times \mathbb{C}^{m-1}$
of parabolic stable pairs 
in the following way: 
\begin{align*}
( b_0, b_1, \cdots, b_{m-1}) \mapsto \left(a_0 b_0, a_0 b_1+a_1 b_0, \cdots, 
\sum_{k=0}^{m-1} a_k b_{m-1-k} \right). 
\end{align*}
Hence the action of 
$\Aut(F)$ on the set of parabolic 
stable pairs 
$\mathbb{C}^{\ast} \times \mathbb{C}^{m-1}$
is free and transitive.  
This implies that the moduli space
$M_{\widehat{H}}^{\rm{par}}(m[C], 0)$ consists of one point, 
hence we have $\mathrm{Par}_{\widehat{H}}^{\nu}(mC, 0)=1$
for $1\le m\le n$. 
\end{proof}
We have the following result: 
\begin{thm}\label{cor:ii}
Suppose that $\nu \equiv 1$ and $l=1$.
Let $n$ be the width of $C$, and $f_n(x)$ 
the polynomial (\ref{fn}).  
Then we have the following identity: 
\begin{align}\notag
\DT_{f^{\dag \ast}\omega}^{\nu}(\phi_{\ast}P)
&=\phi_{\ast} \DT_{f^{\ast}\omega}^{\nu}(P) \\
\label{eq:cor}
&\hspace{10mm} \cdot \left\{ q^{\frac{n}{12}}t^{\frac{n}{2}C^{\dag}}
\prod_{m \in \mathbb{Z}_{>0}}f_n(q^{m}t^{C^{\dag}})
\prod_{m \in \mathbb{Z}_{\ge 0}} f_n(q^{m} t^{-C^{\dag}})
 \right\}^{P \cdot C}. 
\end{align}
\end{thm}
\begin{proof}
We take a smooth divisor $\widehat{H}^{\dag}$ in $\widehat{X}^{\dag}$
such that $\widehat{H}^{\dag}\cap C^{\dag}$
consists of one point. 
By Proposition~\ref{lem:psi}, Theorem~\ref{thm:main} and (\ref{Par=N}), 
we have 
\begin{align*}
\phi_{\ast} \DT_{f^{\ast}\omega}^{\nu}(P) &=
\DT_{f^{\dag \ast}\omega}^{\nu}(\phi_{\ast}P) \\
&\cdot
 \left(q^{\frac{n}{12}}t^{\frac{n}{2}C^{\dag}}\right)^{-P \cdot C}
\left\{\prod_{\mu>0} \mathrm{Par}_{\widehat{H}^{\dag}}^{\nu, \mu}(q, t) 
\prod_{\mu \ge 0} \mathrm{Par}_{\widehat{H}^{\dag}}^{\nu, \mu}(q, t^{-1})
\right\}^{\phi_{\ast}P \cdot C^{\dag}}.
\end{align*}
Applying Lemma~\ref{lem:cup} and 
Lemma~\ref{lem:compv}, we obtain the desired result. 
\end{proof}
\begin{rmk}
Noting that
$f_n(x)=\prod_{j=1}^{n} (1-\xi^j x)$
for $\xi=e^{\frac{2\pi i}{n+1}}$, 
(\ref{theta:pro}), and 
$\prod_{j=1}^{n}(-i \xi^{-\frac{j}{2}})=(-1)^n$, 
the formula (\ref{eq:cor})
is also written as 
\begin{align*}
\DT_{f^{\dag \ast}\omega}^{\nu}(\phi_{\ast}P)=
\phi_{\ast} \DT_{f^{\ast}\omega}^{\nu} \cdot \prod_{j=1}^{n}
\left\{-\eta(q)^{-1} \vartheta_{1, 1}(q, \xi^j t^{C^{\dag}})
   \right\}^{P \cdot C}. 
\end{align*}
\end{rmk}

\subsection{A version with fixed supports}\label{subsec:local}
We can similarly prove a variant of Theorem~\ref{thm:intro}, 
Theorem~\ref{cor:ii}
for DT type invariants counting semistable 
sheaves supported on 
a fixed effective divisor $S \subset X$. 
We denote by $\Coh_S(X)$ the subcategory 
of $\Coh_{\le 2}(X)$ consisting of sheaves supported on $S$. 
We set $\cC_{f^{\ast}\omega}^{\mu, S}$ to be
\begin{align*}
 \cC_{f^{\ast}\omega}^{\mu, S}
\cneq \cC_{f^{\ast}\omega}^{\mu} \cap \Coh_S(X).
\end{align*}
For $v\in \pGamma_{f^{\ast}\omega}^{\mu}$, we define
$\DT_{f^{\ast}\omega}^{\nu, S}(v) \in \mathbb{Q}$
as follows: 
\begin{align}\label{Pi:DT2}
\Pi^{\nu}\left(\epsilon_{\cC_{f^{\ast}\omega}^{\mu, S}}\right)
= \sum_{v \in \pGamma_{f^{\ast}\omega}^{\mu} }
(-1)^{\epsilon(\nu)} \DT^{\nu, S}_{f^{\ast}\omega}(v) \cdot c_v. 
\end{align} 
Similarly to $\DT_{f^{\ast}\omega}^{\nu}(P)$, we set
\begin{align*}
\DT_{f^{\ast}\omega}^{\nu, S}(P) \cneq 
\sum_{n, \beta} \DT_{f^{\ast}\omega}^{\nu, S}(P, -\beta, -n)q^n t^{\beta}
\end{align*}
for $P \in H^2(X)$. Note that if $S$ is an irreducible divisor 
then $\DT_{f^{\ast}\omega}^{\nu, S}(P)$ is non-zero only if 
$P$ is a positive multiple of the cohomology class of $S$. 
Let $S^{\dag} \subset X^{\dag}$ be the strict
transform of $S$, and 
$f_n(x)$ the polynomial as in (\ref{fn}). 
We have the following result: 
\begin{thm}\label{cor:local}
(i) If $\nu=\chi_B$, 
we have the following formula: 
\begin{align*}
&\DT_{f^{\dag \ast} \omega}^{\nu, S^{\dag}}(\phi_{\ast}P)
=\phi_{\ast}\DT_{f^{\ast}\omega}^{\nu, S}(P)
 \\
&\hspace{20mm}\cdot 
\prod_{j=1}^{l} \left\{ 
i^{jP \cdot C-1}
\eta(q)^{-1}\vartheta_{1, 1}(q, ((-1)^{\phi_{\ast}P}t)^{jC^{\dag}}) \right\}^{jn_j P \cdot C}. 
\end{align*}

(ii) If $\nu \equiv 1$ and $l=1$, let 
$n$ be the width of $C$, and 
set $\xi=e^{\frac{2\pi i}{n+1}}$. We have the 
following formula:
\begin{align}\notag
\DT_{f^{\dag \ast}\omega}^{\nu, S^{\dag}}(\phi_{\ast}P)
&=\phi_{\ast} \DT_{f^{\ast}\omega}^{\nu, S}(P) \\
\label{nu1}
&\hspace{10mm} \cdot \left\{ q^{\frac{n}{12}}t^{\frac{n}{2}C^{\dag}}
\prod_{m \in \mathbb{Z}_{>0}}f_n(q^{m}t^{C^{\dag}})
\prod_{m \in \mathbb{Z}_{\ge 0}} f_n(q^{m} t^{-C^{\dag}})
 \right\}^{P \cdot C}. 
\end{align}
\end{thm}
\begin{proof}
Let $\pB_{f^{\ast}\omega}^{\mu, S} \subset \pB_{f^{\ast}\omega}^{\mu}$ be the 
subcategory consisting of objects $E \in \pB_{f^{\ast}\omega}^{\mu}$ 
such that $E|_{X \setminus \Ex(f)}$ is supported on $S$. 
It is easy to check that $\pB_{f^{\ast}\omega}^{\mu, S}$
is an abelian subcategory of $\pB_{f^{\ast}\omega}^{\mu}$. 
We claim that the following equality holds
in $\widehat{H}(\pB_{f^{\ast}\omega}^{\mu})$: 
\begin{align}\label{equal:deltaS}
\delta_{\pB_{f^{\ast}\omega}^{\mu, S}}
= \delta_{\pF} \ast \delta_{\cC_{f^{\ast}\omega}^{\mu, S}} 
\ast \delta_{\pT[-1]}. 
\end{align}
Similarly to Proposition~\ref{prop:delta}, 
it is enough to show the following: 
for any $E \in \pB_{f^{\ast}\omega}^{\mu, S}$, 
if we take a filtration (\ref{filt:B}), then 
$F_2 \in \cC_{f^{\ast}\omega}^{\mu, S}$ holds. 
The condition 
$E \in \pB_{f^{\ast}\omega}^{\mu, S}$
implies that $F_2|_{X \setminus \Ex(f)}$ is supported on $S$. 
On the other hand, the sheaf $F_2$ is 
$\hat{\mu}_{f^{\ast}\omega}$-semistable, hence it is pure 
two dimensional. 
Therefore $F_2$ must be supported on $S$, 
which implies 
$F_2 \in \cC_{f^{\ast}\omega}^{\mu, S}$. 
Hence the equality (\ref{equal:deltaS}) holds.

Let $\Phi$ be the derived equivalence (\ref{Deq}). 
By Lemma~\ref{Phires}, 
it is obvious that $\Phi$ takes $\oB_{f^{\ast}\omega}^{\mu, S}$ to 
$\iB_{f^{\dag \ast}\omega}^{\mu, S^{\dag}}$.
Hence
we have the equality
\begin{align*}
\Phi^H_{\ast} \delta_{\oB_{f^{\ast}\omega}^{\mu, S}}=\delta_{\iB_{f^{\dag\ast}\omega}^{\mu, S^{\dag}}}.
\end{align*}
Therefore the same argument of Theorem~\ref{thm:main} is applied 
and we obtain the following formula:
\begin{align*}
&\phi_{\ast} \DT_{f^{\ast}\omega}^{\nu, S}(P) =q^{\psi_0(P)}t^{\psi_1(P)}
\DT_{f^{\dag \ast}\omega}^{\nu, S^{\dag}}(\phi_{\ast}P) \\
& \cdot \prod_{\begin{subarray}{c}
n> 0, \beta^{\dag}>0 \\
f^{\dag}_{\ast}\beta^{\dag}=0
\end{subarray}}
\exp 
\left( N_{n, \beta^{\dag}}^{\nu} q^n t^{\beta^{\dag}} 
\right)^{\langle \phi_{\ast}P, \beta^{\dag} \rangle}
\prod_{\begin{subarray}{c}
n\ge 0, \beta^{\dag}>0 \\
f^{\dag}_{\ast}\beta^{\dag}=0
\end{subarray}} 
\exp 
\left( N_{n, \beta^{\dag}}^{\nu} q^n 
t^{-\beta^{\dag}} \right)^{\langle \phi_{\ast}P, \beta^{\dag} \rangle}. 
\end{align*}
Then the computations of the error term in the previous
subsections give the desired result. 
\end{proof}

\section{Blow-up formula for DT type invariants on canonical line bundles on surfaces}\label{section:blow}
Let $S$ be a smooth projective surface, and 
\begin{align*}
\pi \colon \omega_S \to S
\end{align*}
the total space of the canonical line bundle of 
$S$, 
which is a non-compact Calabi-Yau 3-fold. 
The purpose of this section is to 
compare DT type invariants on $\omega_S$
under a blow-up $g \colon S^{\dag} \to S$ at
a point $p\in S$. 
In what follows, we regard $S, S^{\dag}$ as subvarieties of 
$\omega_S, \omega_{S^{\dag}}$ by the zero sections. 
\subsection{DT type invariants on canonical line bundles}\label{subsec:DTtype}
Let $\Coh_{\rm{c}}(\omega_S)$ be the abelian category of 
coherent sheaves on $\omega_S$ with compact supports. 
We have the following inclusions of abelian subcategories
\begin{align*}
\Coh(S) \subset \Coh_S(\omega_S) \subset \Coh_{\rm{c}}(\omega_S). 
\end{align*}
For $E \in \Coh_{\rm{c}}(\omega_S)$, 
we set 
\begin{align*}
\ch(\pi_{\ast} E)=(r, l, s) \in H^{0}(S) \oplus H^2(S) \oplus H^4(S).
\end{align*}
Let $L$ be an ample divisor on $S$. We 
define the slope $\mu_L(E)$ to be 
\begin{align*}
\mu_L(E) \cneq \frac{l \cdot L}{r} \in \mathbb{Q} \cup \{\infty\}.  
\end{align*}
Here we set $\mu_L(E)=\infty$ if $r=0$. 
The above slope function defines the $\mu_L$-stability 
on $\Coh_{\rm{c}}(\omega_S)$, 
which restricts to the usual 
$L$-slope stability on $\Coh(S)$. 
Let 
$\cC oh_{\rm{c}}(\omega_S)$ be the stack of 
all the objects in $\Coh_{c}(\omega_S)$. 
The stack $\cC oh_{\rm{c}}(\omega_S)$ is an open substack of 
the algebraic stack of 
all the objects in $\Coh(X)$ for any projective compactification
$\omega_S \subset X$.
Hence $\cC oh_{\rm{c}}(\omega_S)$ is 
an algebraic stack locally of finite type. 
Similarly to Subsections~\ref{subsec:Hall}, ~\ref{subsec:int}, 
we can define the Hall algebra $H(\Coh_{\rm{c}}(\omega_S))$
of $\Coh_{\rm{c}}(\omega_S)$, 
the Lie subalgebra $H^{\rm{Lie}}(\Coh_{\rm{c}}(\omega_S))$
of virtual indecomposable objects, and 
the unweighted integration map 
\begin{align}\label{Lie:pi}
\Pi^{\nu \equiv 1} \colon 
H^{\rm{Lie}}(\Coh_{\rm{c}}(\omega_S))
\to \bigoplus_{(r, l, s) \in H^{\ast}(S)}
\mathbb{Q} \cdot c_{(r, l, s)}. 
\end{align}
For $\mu \in \mathbb{Q}$, let $\cC^{\mu}_{L} \subset \cC oh_{\rm{c}}(\omega_S)$
be the substack corresponding to $\mu_L$-semistable 
sheaves $E \in \Coh_S(\omega_S)$ with slope $\mu$. 
The map (\ref{Lie:pi})
extends to appropriate completions of both sides, and
the invariant  
$\DT_L^{\chi}(r, l, s) \in \mathbb{Q}$ is defined by
\begin{align*}
\Pi^{\nu \equiv 1}(\epsilon_{\cC^{\mu}_L})=\sum_{l \cdot L/r =\mu}
\DT_L^{\chi}(r, l, s) \cdot c_{(r, l, s)}. 
\end{align*}
\begin{rmk}
Let $\mM_L(r, l, s)$ be the moduli stack of 
$\mu_L$-semistable sheaves $E\in \Coh_S(\omega_S)$
with $\ch(\pi_{\ast}E)=(r, l, s)$. 
By the same reason of Remark~\ref{rmk:GIT}, 
if any $[E] \in \mM_L(r, l, s)$ is $\mu_L$-stable, then 
$\mM_L(r, l, s)$ is a $\mathbb{C}^{\ast}$-gerbe
over a projective scheme $M_L(r, l, s)$.
The invariant
$\DT_L^{\chi}(r, l, s)$
coincides with the naive Euler characteristic of 
$M_L(r, l, s)$. 
\end{rmk}
For $(r, l, s) \in H^{\ast}(S^{\dag})$, 
we can similarly define the invariant
$\DT_{g^{\ast}L}^{\chi}(r, l, s) \in \mathbb{Q}$
by replacing $(S, L)$ by $(S^{\dag}, g^{\ast}L)$
in the above construction.

\subsection{Blow-up and 3-fold flop}
The following lemma 
relates a blow-up of a surface
with a 3-fold flop: 
\begin{lem}\label{lem:com}
There exist smooth projective 3-folds $X$, $X^{\dag}$
and a flop diagram (\ref{flop:dia})
satisfying the following conditions: 
\begin{itemize}
\item Both of the
 exceptional locus $C=\Ex(f)$, $C^{\dag} =\Ex(f^{\dag})$
are 
irreducible 
$(-1, -1)$-curves. 
\item There are closed embeddings 
\begin{align*}
i \colon S \hookrightarrow X, \quad 
i^{\dag} \colon S^{\dag} \hookrightarrow X^{\dag}
\end{align*}
such that 
$S \cap C$ consists of one point, 
the strict transform of $S$
in $X^{\dag}$ 
coincides with $S^{\dag}$, 
and $C^{\dag} \subset S^{\dag}$ coincides with 
the exceptional locus of $g \colon S^{\dag} \to S$. 
Moreover $f^{\dag}(S^{\dag})=S$, 
$f^{\dag}|_{S^{\dag}}=g$ and 
$kL$ for $k\gg 0$
 extends to an ample divisor $L'$ on $Y$.  
\item There are open neighborhoods 
\begin{align}\label{open}
S \subset X_{0}, \quad S^{\dag} \subset X_0^{\dag}
\end{align}
which are isomorphic to $\omega_S$, $\omega_{S^{\dag}}$ 
respectively, such that the embeddings (\ref{open})
are identified with the zero sections. 
\end{itemize}
\end{lem}
\begin{proof}
Let $C^{\dag} \subset S^{\dag}$ be the exceptional locus of 
$g$, and $m_p \subset \oO_S$ be the ideal sheaf of $p$. 
We set $U^{\dag}=\omega_{S^{\dag}}$ and 
consider the following commutative diagram
\begin{align*}
\xymatrix{
U^{\dag}= \sS pec_{\oO_{S^{\dag}}} \left(\bigoplus_{k\ge 0} 
\omega_{S^{\dag}}^{\otimes -k}\right) \ar[r]^{\hspace{-2mm}h^{\dag}} \ar[d]_{\pi^{\dag}} & V=\sS pec_{\oO_S} \left( \bigoplus_{k\ge 0} m_p^{k} \otimes
\omega_{S}^{\otimes -k} \right) \ar[d]^{\pi_0} \\
S^{\dag} \ar[r]^{\hspace{-2mm}g} & S. 
}
\end{align*}
Here the morphism $h^{\dag}$ is induced by 
$g_{\ast}\omega_{S^{\dag}}^{-k} \cong m_p^k \otimes \omega_S^{-k}$, 
and it is a birational morphism. 
We embed $S^{\dag}$ into $U^{\dag}$ by the zero section of 
$\pi^{\dag}$. 
It is easy to check that the curve 
$C^{\dag} \subset S^{\dag}$ is a $(-1, -1)$-curve in 
$U^{\dag}$, 
which coincides with the exceptional 
locus of $h^{\dag}$, and $h^{\dag}(C^{\dag})$ is an 
ordinary double point in $V$. 
Hence $h^{\dag}$ is a 3-fold flopping contraction. 
Let $h \colon U \to V$ be the flop of 
$h^{\dag}$, and $C \subset U$ the exceptional locus of $h$. 
Note that $U$ and $U^{\dag}$ are 
related by the diagram
$U \leftarrow W \to U^{\dag}$, 
where the left morphism is a blow-up at $C$, 
and the right morphism is a 
blow-up at $C^{\dag}$. 
Hence the strict transform of $S^{\dag}$ in $U$ 
coincides with $S$, which intersects with $C$ at 
$p$. 

We show that $U$ contains an open neighborhood of $S$ 
which is isomorphic
to $\omega_S$. 
Let us consider the divisor $\pi^{\dag -1}(C^{\dag})$ in 
$U^{\dag}$, and its strict transform 
$D \subset U$.
Note that $S \cap D =\emptyset$, and $D\cap C=\{q\}$ with $p\neq q$. 
We claim that $U \setminus D$ is isomorphic to $\omega_S$, 
such that $S \subset U \setminus D$ is identified with the zero section. 
Note that $U\setminus D$ is set theoretically 
written as $\pi^{-1}(S \setminus \{p\}) \sqcup (C \setminus \{q\})$, 
where $\pi \colon \omega_S \to S$ is the projection. 
We consider the map from $U \setminus D$ to $S$ by 
sending $x \in \pi^{-1}(S\setminus \{p\})$ to $\pi(x)$
and $x \in C \setminus \{q\}$ to 
$p$. It is easy to check that  
this map is a Zariski locally trivial $\mathbb{A}^1$-fibration, hence
$U\setminus D$ is 
a total space of some line bundle $\lL$ on $S$. 
Since $\lL$ is isomorphic to $\omega_S$ outside $p$, 
it must be isomorphic to $\omega_S$. 
Hence $U \setminus D$ is isomorphic to $\omega_S$. 

We consider the $\mathbb{P}^1$-bundle over $S^{\dag}$ given by 
\begin{align*}
\overline{\pi}^{\dag} \colon 
X^{\dag}=\mathbb{P}(\oO_{S^{\dag}} \oplus \omega_{S^{\dag}}) \to S^{\dag}. 
\end{align*}
Note that
$X^{\dag}$ is a projective compactification of $U^{\dag}$.
Let $E\cneq X^{\dag} \setminus U^{\dag}$ be the boundary divisor.
By the base point free theorem, the divisor 
$k\overline{\pi}^{\dag \ast} g^{\ast}L+E$
is globally generated for $k\gg 0$. 
The resulting morphism $f^{\dag} \colon X^{\dag} \to Y$ is a 
birational morphism whose exceptional locus coincides with $C^{\dag}$. 
Hence $f^{\dag}$ is a 3-fold flopping contraction and $Y$
is a projective compactification of $V$. 
By the construction,  
there is an ample divisor $L'$ on $Y$
such that 
$k\overline{\pi}^{\dag \ast} g^{\ast}L+E=f^{\dag \ast}L'$, 
which implies $L'|_{S}=kL|_{S}$. 
By taking the flop of $f^{\dag}$, we obtain a desired flop diagram. 
\end{proof}

\subsection{Blow-up formula}
We compare the DT type 
invariants 
$\DT_{L}^{\chi}(\ast)$
and $\DT_{g^{\ast}L}^{\chi}(\ast)$
in Subsection~\ref{subsec:DTtype}
using Theorem~\ref{cor:local}
and Lemma~\ref{lem:com}. 
Let $\omega_S \subset X$ be a compactification 
as in Lemma~\ref{lem:com}. 
By the Grothendieck 
Riemann-Roch theorem, an object $E \in \Coh_S(\omega_S)$
satisfies $\ch(\pi_{\ast}E)=(r, l, s)$
if and only if 
\begin{align*}
v(E)=\left(rS, i_{\ast}\left(l-\frac{r}{2}K_S \right), 
\frac{rK_S^2}{12} -\frac{K_S l}{2} +\frac{r\chi(\oO_S)}{2}+s  \right)
\end{align*}
in $H^{\ge 2}(X)$. 
Hence $E \in \Coh_S(\omega_S)$ is $\mu_L$-(semi)stable 
if and only if $E \in \Coh_S(X)$ and it is
$\hat{\mu}_{f^{\ast}L'}$-(semi)stable, where 
$L'$ is an ample divisor on $Y$
as in Lemma~\ref{lem:com}. 
We have the following identity for $\nu \equiv 1$: 
\begin{align*}
\DT_{L}^{\chi}(r, l, s) =
\DT_{f^{\ast}L'}^{\nu, S}\left(rS, i_{\ast}\left(l-\frac{rK_S}{2} \right), 
\frac{rK_S^2}{12} -\frac{K_S l}{2}+\frac{r\chi(\oO_S)}{2}+s \right). 
\end{align*}
Here we have used the notation in Subsection~\ref{subsec:local}.
Similarly we have the identity for $\nu \equiv 1$:
\begin{align*}
\DT_{g^{\ast}L}^{\chi}(r, l, s) =
\DT_{f^{\dag \ast}L}^{\nu, S^{\dag}}\left(rS^{\dag}, 
i_{\ast}^{\dag}\left(l-\frac{rK_{S^{\dag}}}{2} \right), 
\frac{rK_{S^{\dag}}^2}{12} -\frac{K_{S^{\dag}} l }{2}
+\frac{r\chi(\oO_{S^{\dag}})}{2}+s \right). 
\end{align*}
We have the following result:
\begin{thm}\label{thm:blow}
For fixed $r \in \mathbb{Z}_{\ge 1}$ and 
$l\in H^2(S)$, we have the following formula:
\begin{align*}
&\sum_{s, a} \DT^{\chi}_{g^{\ast}L}(r, g^{\ast}l-aC^{\dag}, -s)
q^{\frac{r}{12}+\frac{a}{2}+s}t^{a+\frac{r}{2}} \\
&\hspace{30mm} =\sum_{s} \DT^{\chi}_{L}(r, l, -s)q^{s} 
\cdot \eta(q)^{-r} 
\vartheta_{1, 0}(q, t)^{r}. 
\end{align*}
\end{thm}
\begin{proof}
Let $\phi \colon X \dashrightarrow X^{\dag}$ be a flop
as in Lemma~\ref{lem:com}. 
Note that $\phi_{\ast} \colon H^4(X) \to H^4(X^{\dag})$
takes elements of the form $i_{\ast}l$ 
for $l \in H^2(S)$
to $i^{\dag}_{\ast}g^{\ast}l$. 
Noting that $H^2(S^{\dag})=g^{\ast}H^2(S) \oplus \mathbb{Q}\cdot [C^{\dag}]$
and applying Theorem~\ref{cor:local} for $\nu \equiv 1$, 
we obtain 
\begin{align*}
&\sum_{l, s}
\DT_{L}^{\chi}(r, l, -s)
q^{-\frac{rK_S^2}{12}+\frac{K_S l}{2} -\frac{r\chi(\oO_S)}{2}+s}
t^{g^{\ast}\left(-l+\frac{rK_S}{2}\right)} \\
&=\phi_{\ast}\left( \sum_{l, s}
\DT_{f^{\ast}L'}^{\nu, S}\left(rS, i_{\ast}\left(l-\frac{rK_S}{2} \right), 
\frac{rK_S^2}{12} -\frac{K_S l}{2}+\frac{r\chi(\oO_S)}{2}-s \right) \right. \\
& \hspace{75mm}
\left.
\cdot q^{-\frac{rK_S^2}{12}+\frac{K_S l}{2} -\frac{r\chi(\oO_S)}{2}+s}
t^{-l+\frac{rK_S}{2}}
 \right) \\
&= \sum_{l, s, a}
\DT_{f^{\dag \ast}L'}^{\nu, S^{\dag}}
\left(rS^{\dag}, i_{\ast}^{\dag}
\left(g^{\ast}l-aC^{\dag}-\frac{rK_{S^{\dag}}}{2} 
\right), 
\frac{rK_{S^{\dag}}^2}{12} -\frac{K_{S^{\dag}}}{2}(g^{\ast}l-aC^{\dag}) 
\right.\\
&\hspace{20mm}\left.+\frac{r\chi(\oO_{S^{\dag}})}{2}-s
 \right) \cdot 
q^{-\frac{rK_{S^{\dag}}^2}{12} +\frac{K_{S^{\dag}}}{2}(g^{\ast} l -aC^{\dag})
-\frac{r\chi(\oO_{S'})}{2}+s}
t^{-g^{\ast}l  +aC^{\dag}+\frac{rK_{S^{\dag}}}{2}} \\
&\hspace{85mm} \cdot
\{\eta(q)^{-1}\vartheta_{1, 1}(q, -t^{C^{\dag}}) \}^{-r} \\
&=\sum_{l, s, a}\DT_{g^{\ast}L}^{\chi}(r, g^{\ast}l-aC^{\dag}, -s) 
q^{-\frac{rK_S^2}{12}+\frac{K_S l}{2} -\frac{r\chi(\oO_S)}{2}+\frac{r}{12}+\frac{a}{2}+s}
\\
&\hspace{60mm}
t^{g^{\ast}\left(-l+\frac{rK_S}{2}\right)+
\left(a+\frac{r}{2}\right)C^{\dag}} 
\cdot \eta(q)^{r}
\vartheta_{1, 0}(q, t^{C^{\dag}})^{-r}.
\end{align*}
Therefore we obtain the desired result. 
\end{proof}

\begin{exam}\label{exam:r2}
Suppose that $\nu\equiv 1$ and 
$r=2$. 
We take $a \in \{0, 1\}$ and 
set $\widetilde{l}=g^{\ast}l -aC^{\dag}$. 
Then the result of Theorem~\ref{thm:blow}
shows that 
\begin{align*}
\sum_{s}\DT_{g^{\ast}L}^{\chi}(2, \widetilde{l}, -s) 
q^{s +\frac{\widetilde{l}^2}{4}}
= q^{\frac{1}{12}} \frac{\vartheta_a(q)}{\eta(q)^2}
\sum_{s} \DT_{L}^{\chi}(2, l, -s) q^{s+ \frac{l^2}{4}}.
\end{align*}
Here $\vartheta_a(q)= 
\sum_{k\in\mathbb{Z}}q^{\left( k+ \frac{a}{2} \right)^2}$. 
The above formula coincides with the blow-up formula 
obtained by Li-Qin~\cite[Theorem~A]{WZ}. 
\end{exam}

\providecommand{\bysame}{\leavevmode\hbox to3em{\hrulefill}\thinspace}
\providecommand{\MR}{\relax\ifhmode\unskip\space\fi MR }
\providecommand{\MRhref}[2]{%
  \href{http://www.ams.org/mathscinet-getitem?mr=#1}{#2}
}
\providecommand{\href}[2]{#2}

Kavli Institute for the Physics and 
Mathematics of the Universe, University of Tokyo,
5-1-5 Kashiwanoha, Kashiwa, 277-8583, Japan.

\textit{E-mail address}: yukinobu.toda@ipmu.jp


\begin{thebibliography}{Tod13b}

\bibitem[Beh09]{Beh}
K.~Behrend, \emph{Donaldson-{T}homas invariants via microlocal geometry},
  Ann.~of Math \textbf{170} (2009), 1307--1338.

\bibitem[BKL01]{BKL}
J.~Bryan, S.~Katz, and N.~C. Leung, \emph{Multiple covers and integrality
  conjecture for rational curves on {C}alabi-{Y}au threefolds}, J.~Algebraic
  Geom.~ \textbf{10} (2001), 549--568.

\bibitem[Bri02]{Br1}
T.~Bridgeland, \emph{Flops and derived categories}, Invent. Math \textbf{147}
  (2002), 613--632.

\bibitem[Bri07]{Brs1}
\bysame, \emph{Stability conditions on triangulated categories}, Ann.~of Math
  \textbf{166} (2007), 317--345.

\bibitem[Bri11]{BrH}
\bysame, \emph{Hall algebras and curve-counting invariants},
  J.~Amer.~Math.~Soc.~ \textbf{24} (2011), 969--998.

\bibitem[Cal]{Cala}
J.~Calabrese, \emph{Donaldson-{T}homas invariants on {F}lops}, preprint,
  arXiv:1111.1670.

\bibitem[C{\u{a}}l05]{Cal2}
A.~C{\u{a}}ld{\u{a}}raru, \emph{The {M}ukai pairing, {II}:~{T}he
  {H}ochschild-{K}ostant-{R}osenberg isomorphism}, Advances in Math.~
  \textbf{194} (2005), 34--66.

\bibitem[Che02]{Ch}
J-C. Chen, \emph{Flops and equivalences of derived categories for three-folds
  with only {G}orenstein singularities}, J.~Differential.~Geom \textbf{61}
  (2002), 227--261.

\bibitem[dB04]{MVB}
M.~Van den Bergh, \emph{Three dimensional flops and noncommutative rings}, Duke
  Math.~J.~ \textbf{122} (2004), 423--455.

\bibitem[DM]{DM}
F.~Denef and G.~Moore, \emph{Split states, {E}ntropy {E}nigmas, {H}oles and
  {H}alos}, arXiv:hep-th/0702146.

\bibitem[DW]{WM}
W.~Donovan and M.~Wemyss, \emph{Noncommutative deformations and flops},
  preprint, arXiv:1309.0698.

\bibitem[EZ85]{EZ}
M.~Eichler and D.~Zagier, \emph{The theory of {J}acobi forms}, Progress in
  Mathematics, vol.~55, Birkh\"auser Boston, 1985.

\bibitem[Ful]{Fu}
W.~Fulton, \emph{Intersection theory. {S}econd edition}, Ergebnisse der
  Mathematik und ihrer {G}renzgebiete. 3. Folge, vol.~2, Springer-Verlag.

\bibitem[G\"99]{GoTheta}
L.~G\"ottsche, \emph{Theta functions and {H}odge numbers of moduli spaces of
  sheaves on rational surfaces}, Comm.~Math.~Phys.~ \textbf{206} (1999),
  105--136.

\bibitem[G\"09]{Goinv}
\bysame, \emph{Invariants of {M}oduli {S}paces and {M}odular {F}orms},
  Rend.~Istit.~Mat.~Univ.~Trieste \textbf{41} (2009), 55--76.

\bibitem[GSa]{GS3}
A.~Gholampour and A.~Sheshmani, \emph{Donaldson-{T}homas {I}nvariants of
  2-{D}imensional sheaves inside threefolds and modular forms}, preprint,
  arXiv:1309.0050.

\bibitem[GSb]{GS1}
\bysame, \emph{Generalized {D}onaldson-{T}homas {I}nvariants of 2-{D}imensional
  sheaves on local $\mathbb{P}^2$}, preprint, arXiv:1309.0056.

\bibitem[GST]{GS2}
A.~Gholampour, A.~Sheshmani, and R.~P. Thomas, \emph{Counting curves on
  surfaces in {C}alabi-{Y}au 3-folds}, preprint, arXiv:1309.0051.

\bibitem[GSY]{GSY06}
D.~Gaiotto, A.~Strominger, and X.~Yin, \emph{The {M}5-brane elliptic genus:
  {M}odularity and {BPS} states}, arXiv:hep-th/0607010.

\bibitem[GY]{GX07}
D.~Gaiotto and X.~Yin, \emph{Examples of {M}5-{B}rane {E}lliptic {G}enera},
  arXiv:hep-th/0702012.

\bibitem[HL97]{Hu}
D.~Huybrechts and M.~Lehn, \emph{Geometry of moduli spaces of sheaves}, Aspects
  in Mathematics, vol. E31, Vieweg, 1997.

\bibitem[HL12]{HL}
J.~Hu and W.~P. Li, \emph{The {D}onaldson-{T}homas invariants under blowups and
  flops}, J.~Differential.~Geom.~ (2012), 391--411.

\bibitem[HRS96]{HRS}
D.~Happel, I.~Reiten, and S.~O. Smal$\o$, \emph{Tilting in abelian categories
  and quasitilted algebras}, Mem.~Amer.~Math.~Soc, vol. 120, 1996.

\bibitem[Joy07]{Joy2}
D.~Joyce, \emph{Configurations in abelian categories {I}\hspace{-.1em}{I}.
  {R}ingel-{H}all algebras}, Advances in Math \textbf{210} (2007), 635--706.

\bibitem[Joy08]{Joy4}
D.~Joyce, \emph{Configurations in abelian categories {I}\hspace{-.1em}{V}.
  {I}nvariants and changing stability conditions}, Advances in Math
  \textbf{217} (2008), 125--204.

\bibitem[JS12]{JS}
D.~Joyce and Y.~Song, \emph{A theory of generalized {D}onaldson-{T}homas
  invariants}, Mem.~Amer.~Math.~Soc.~ \textbf{217} (2012).

\bibitem[Kat08]{Katz}
S.~Katz, \emph{Genus zero {G}opakumar-{V}afa invariants of contractible
  curves}, J.~Differential.~Geom.~ \textbf{79} (2008), 185--195.

\bibitem[Kaw08]{Kawaflo}
Y.~Kawamata, \emph{Flops connect minimal models}, Publ.~Res.~Inst.~Math.~Sci.~
  \textbf{44} (2008), 419--423.

\bibitem[KM92]{KaMo}
S.~Katz and D.~R. Morrison, \emph{Gorenstein threefold singularities with small
  resolutions via invariant theory for {W}eyl groups}, J.~Algebraic Geom.~
  \textbf{1} (1992), 449--530.

\bibitem[KS]{K-S}
M.~Kontsevich and Y.~Soibelman, \emph{Stability structures, motivic
  {D}onaldson-{T}homas invariants and cluster transformations}, preprint,
  arXiv:0811.2435.

\bibitem[Lan04]{Langer}
A.~Langer, \emph{Semistable sheaves in positive characteristic}, Ann.~of Math.~
  \textbf{159} (2004), 251--276.

\bibitem[Lan09]{Langer2}
\bysame, \emph{Moduli spaces of sheaves and principal {$G$}-bundles},
  Proc.~Sympos.~Pure Math.~ \textbf{80} (2009), 273--308.

\bibitem[Lie06]{LIE}
M.~Lieblich, \emph{Moduli of complexes on a proper morphism}, J.~Algebraic
  Geom.~ \textbf{15} (2006), 175--206.

\bibitem[LQ99]{WZ}
W.~P. Li and Z.~Qin, \emph{On blowup formulae for the {$S$}-duality conjecture
  of {V}afa and {W}itten}, Invent.~Math.~ \textbf{136} (1999), 451--482.

\bibitem[Nag13]{Nagao}
K.~Nagao, \emph{Donaldson-{T}homas theory and cluster algebras}, Duke Math.~J.~
  \textbf{162} (2013), 1313--1367.

\bibitem[Nis]{Nishi}
T.~Nishinaka, \emph{Multiple {D}4-{D}2-{D}0 on the {C}onifold and
  {W}all-crossing with the {F}lop}, preprint, arXiv:1010.6002.

\bibitem[NN11]{NN}
K.~Nagao and H.~Nakajima, \emph{Counting invariant of perverse coherent sheaves
  and its wall-crossing}, Int.~Math.~Res.~Not.~ (2011), 3855--3938.

\bibitem[NY]{NiYa}
T.~Nishinaka and S.~Yamaguchi, \emph{Wall-crossing of {D}4-{D}2-{D}0 and flop
  of the conifold}, preprint, arXiv:1007.2731.

\bibitem[OSV04]{OSV}
H.~Ooguri, A.~Strominger, and C.~Vafa, \emph{Black hole attractors and the
  topological string}, Phys.~Rev.~D \textbf{70} (2004), arXiv:hep-th/0405146.

\bibitem[Rei]{Rei}
M.~Reid, \emph{Minimal models of canonical 3-folds}, Algebraic Varieties and
  Analytic Varieties (S.~Iitaka, ed), Adv. Stud. Pure Math, Kinokuniya, Tokyo,
  and North-Holland, Amsterdam \textbf{1}, 131--180.

\bibitem[Tho00]{Thom}
R.~P. Thomas, \emph{A holomorphic {C}asson invariant for {C}alabi-{Y}au 3-folds
  and bundles on ${K3}$-fibrations}, J.~Differential.~Geom \textbf{54} (2000),
  367--438.

\bibitem[Toda]{Todmu}
Y.~Toda, \emph{Multiple cover formula of generalized {DT} invariants {II}:
  {J}acobian localizations}, preprint, arXiv:1108.4993.

\bibitem[Todb]{TodS2}
\bysame, \emph{S-duality for surfaces with {$A_n$}-type singularities},
  preprint, arXiv:1312.2300.

\bibitem[Tod08a]{ToBPS}
\bysame, \emph{Birational {C}alabi-{Y}au 3-folds and {BPS} state counting},
  Communications in Number Theory and Physics \textbf{2} (2008), 63--112.

\bibitem[Tod08b]{Tst3}
\bysame, \emph{Moduli stacks and invariants of semistable objects on {K}3
  surfaces}, Advances in Math \textbf{217} (2008), 2736--2781.

\bibitem[Tod10]{Tcurve1}
\bysame, \emph{Curve counting theories via stable objects~{I}: {DT/PT}
  correspondence}, J.~Amer.~Math.~Soc.~ \textbf{23} (2010), 1119--1157.

\bibitem[Tod13a]{TodBG}
\bysame, \emph{Bogomolov-{G}ieseker type inequality and counting invariants},
  Journal of Topology \textbf{6} (2013), 217--250.

\bibitem[Tod13b]{Tcurve2}
\bysame, \emph{Curve counting theories via stable objects~{II}. {DT}/nc{DT}
  flop formula}, J.~Reine Angew.~Math.~ \textbf{675} (2013), 1--51.

\bibitem[Tod13c]{Todext}
\bysame, \emph{Stability conditions and extremal contractions}, Math.~Ann.~
  \textbf{357} (2013), 631--685.

\bibitem[Tod14]{Todpara}
\bysame, \emph{Multiple cover formula of generalized {DT} invariants {I}:
  parabolic stable pairs}, Adv.~Math.~ \textbf{257} (2014), 476--526.

\bibitem[VW94]{VW}
C.~Vafa and E.~Witten, \emph{A {S}trong {C}oupling {T}est of {S}-{D}uality},
  Nucl.~Phys.~B \textbf{431} (1994).

\bibitem[Yos96]{Yo1}
K.~Yoshioka, \emph{Chamber structure of polarizations and the moduli space of
  rational elliptic surfaces}, Int.~J.~Math.~ \textbf{7} (1996), 411--431.

\end{thebibliography}
\end{document}